\documentclass{article}
\pdfoutput=1
\usepackage[margin=1.25in]{geometry}
\usepackage{amsmath,amssymb,amsthm,bbm,multirow,asymptote,graphicx,theoremref,mathtools,xcolor}
\usepackage{authblk}

\numberwithin{equation}{section}
\theoremstyle{plain}
\newtheorem{theorem}{Theorem}[section]
\newtheorem*{theorem*}{Theorem}
\newtheorem{lemma}[theorem]{Lemma}
\theoremstyle{definition}

\newtheorem*{definition*}{Definition}
\theoremstyle{remark}
\newtheorem{remark}[theorem]{Remark}

\DeclareRobustCommand{\sumstack}[2]{{\begingroup#1\endgroup\@@over#2}}

\newcommand{\E}{\mathbb{E}}
\renewcommand{\P}{\mathbb{P}}
\newcommand{\Q}{\mathbb{Q}}
\newcommand{\di}{\mathrm{d}}
\newcommand{\iden}{\boldsymbol{1}}
\newcommand{\half}{\dfrac{1}{2}}
\newcommand{\shalf}{\frac{1}{2}}

\newcommand{\R}{\mathbb{R}}

\newcommand{\dx}{\mathrm{d}x}

\newcommand{\F}{\mathcal{F}}
\newcommand{\Ep}[1]{\E\left[#1\right]}
\newcommand{\Eps}[2]{\E^{#1}\left[#2\right]}
\newcommand{\Prob}[1]{\P\left(#1\right)}

\newcommand{\Probc}[2]{\P\left(#1| #2\right)}
\newcommand{\tauN}[1]{\tau^N_{#1}}
\newcommand{\tauNk}{\tau^N_k}
\newcommand{\pjt}{p_{j,t}}
\newcommand{\qjt}{q_{j,t}}
\newcommand{\psub}[1]{p_{#1}}
\newcommand{\qsub}[1]{q_{#1}}

\newcommand{\ttau}{\tilde{\tau}}
\newcommand{\e}{\mathrm{e}}
\newcommand{\eps}{\varepsilon}

\newcommand{\core}{\mathrm{core}}
\newcommand{\dom}{\mathrm{dom}}
\newcommand{\cont}{\mathrm{cont}}

\begin{document}

\title{Discretisation and Duality of Optimal Skorokhod Embedding Problems}

\author{Alexander M. G. Cox}
\author{Sam M. Kinsley}
\affil{Department of Mathematical Sciences, University of Bath, U.K.}

\maketitle

\begin{abstract}
We prove a strong duality result for a linear programming problem which has the interpretation of being a discretised optimal Skorokhod embedding problem, and we recover this continuous time problem as a limit of the discrete problems. With the discrete setup we show that for a suitably chosen objective function, the optimiser takes the form of a hitting time for a random walk. In the limiting problem we then reprove the existence of the Root, Rost, and cave embedding solutions of the Skorokhod embedding problem. 

The main strength of this approach is that we can derive properties of the discrete problem more easily than in continuous time, and then prove that these properties hold in the limit. For example, the strong duality result gives dual optimisers, and our limiting arguments can be used to derive properties of the continuous time dual functions, known to represent a superhedging portfolio.
\end{abstract}

\section{Introduction}
For a Brownian motion $B$ and a centered probability distribution $\mu$ on $\R$ with finite second moment, the Skorokhod embedding problem is to find a stopping time $\tau$ such that
\begin{equation} \label{SEP}
B_{\tau}\sim\mu, \text{ and } (B_{t\wedge\tau})_{t\geq0} \text{ is UI.} \tag{SEP}
\end{equation}
The problem was first introduced, and solved, by Skorokhod \cite{Skorokhod:1961aa,Skorokhod:1965aa} and has since proved important for a range of applications in probability theory and mathematical finance. There have been numerous solutions since that of Skorokhod, and we refer the reader to \cite{Obloj:2004aa} for a survey of all solutions known at the time. Many of these solutions also have some nice optimality properties. For example, two embeddings of particular importance in this paper are the so-called Root and Rost embeddings, which are the stopping times with minimal/maximal variance respectively that solve the Skorokhod embedding problem. Both of these stopping times are actually hitting times of regions in space by a Brownian motion, and these are the types of solution we focus on here. In this paper we consider optimal Skorokhod embedding problems of the form
\begin{equation*} \label{OptSEP}\tag{OptSEP}
\sup_{\tau} \Ep{F(W_{\tau},\tau)} \quad \text{over solutions of \eqref{SEP}}
\end{equation*}
for functions $F$ where it is known that the optimiser is a certain type of hitting time. We discretise this problem to get an infinite linear programming problem which has a well-defined dual and where we are able to prove a strong duality result. For the optimal primal solution we can show that the discrete solution is also of the form of a hitting time, and then in the limit we can recover the continuous time embedding results of, for example, the Root and Rost solutions from \cite{Chacon:1985aa,Root:1969aa,Rost:1971aa,Rost:1976aa}. As well as reproving previous results, we can also look at the limit of the dual optimisers and deduce properties of the continuous time dual functions. Indeed, the inspiration for this work was to find the form of the optimal dual solution to an optimal Skorokhod embedding problem which has the financial interpretation of finding an upper bound on the price of a European call option on a leveraged exchange traded fund (LETF), see \cite{Cox:2016aa}. It is well known that the dual of such a problem corresponds to finding the minimum cost of a self-financing model-independent superhedging portfolio of the option, and we will explain in Section \ref{dualsection} how our discrete dual variables relate to this problem.

The stopping time that maximises \eqref{OptSEP} when the function $F$ is the payoff of a European call option on an LETF is a hitting time where the stopping region, called a $K$-cave barrier, introduced in \cite{Cox:2016aa}, is the combination of a Root barrier and a Rost inverse-barrier. A similar embedding is the cave embedding solution of \cite{Beiglboeck:2013aa}, and we will use this an example throughout this paper since it combines the ideas of both the Root and Rost payoffs and gives some insight into the $K$-cave problem.

In \cite{Beiglboeck:2013aa}, the authors use ideas from martingale optimal transport, which is closely tied to the Skorokhod embedding problem, to show a monotonicity principle which can be used to to prove the existence of all known optimal solutions of \eqref{SEP}. The monotonicity principle for an optimal Skorokhod embedding problem links to the idea of $c$-cyclical monotonicity in optimal transport theory and both of these are based on a notion of \emph{path-swapping}. This is also how we prove the form of the stopping times in the discrete problems in this paper, as can be seen in Section \ref{cave}.

The use of discrete systems in model-independent finance is a common theme, see for example \cite{Acciaio:2013aa, Badikov:2016aa,Bayraktar:2015aa, Beiglboeck:2013ab, Bouchard:2015aa, Hobson:2016ab, Hobson:2016aa, Neuberger:2007aa} among others, however the exact nature of our problem, and the passage from discrete to continuous setups, appears to be novel.

The paper is organised as follows: in Section 2 we set up the discrete linear programming problem by discretising \eqref{OptSEP}. The dual problem is introduced in Section 3 and we give ideas on how this relates to the continuous time dual after proving a strong duality result. In Section 4 we show that the discrete optimal stopping region exhibits the same barrier-type properties as its continuous counterpart. Section 5 focuses on the convergence of the discrete problem back to \eqref{OptSEP}. In Section 5.1 we show that we can discretise a feasible solution of \eqref{OptSEP} and then recover the same stopping time in the limit. Finally in Section 5.2 we prove that our discrete stopping region converges to a continuous stopping region with the same properties, and therefore that the a limit of the discrete optimal solutions is an optimiser of \eqref{OptSEP}.

\section{Discretisation and Primal Formulation}
Consider an atomic measure $\mu$ with bounded support, so in particular it has a highest and lowest atom, i.e. $\exists x_*,x^*$ such that $x^*$ is the smallest $x>0$ such that $\mu((x,\infty))=0$, and $x_*<0$ is the largest $x$ such that $\mu((-\infty,x_*))=0$. To embed this distribution into a Brownian motion (BM) with a uniformly integrable stopping time, we know that we will stop immediately if we hit $x^*$ or $x_*$, so we have absorbing barriers at these levels. We can split the interval $[x_*,x^*]$ into a uniform mesh $(x^N_j)_j$, for $j=\{0,1,\ldots,L(N)\}$, where $N$ is a parameter that we let go to infinity to reduce our mesh size and regain the continuous time case. 

Now run a Brownian motion, $W$, stopping every time we hit a level $x_j^N$, and consider the process formed by this. In other words, consider the process $Y_k^N=W_{\tauNk}$, where $\tauN{0}=0$ and if $W_{\tauNk}=x_j^N$ then $\tauN{k+1}=\inf \left\{ t\geq \tauNk : W_t\in \left\{x^N_{j+1}, x^N_{j-1}\right\} \right\}$ for $k\geq0$. We will later use Donsker's theorem to recover our Brownian motion, and to ensure we can do this we need to choose the correct mesh size. We let $x^N_j=\frac{j}{\sqrt{N}}$ for $j\in\mathcal{J}:=\{\lfloor x_*\sqrt{N}\rfloor, \lfloor x_*\sqrt{N}\rfloor+1, \ldots, \lfloor x^*\sqrt{N} \rfloor\}$. Let $j^N_0:=\lfloor x_*\sqrt{N}\rfloor$, $j^N_1:=\lfloor x_*\sqrt{N}\rfloor+1$, $\ldots$, $j^N_L:=\lfloor x^*\sqrt{N} \rfloor$, where $L\sim\sqrt{N}$, so $\mathcal{J}=\{j^N_0,j^N_1,\ldots,j^N_L\}$. We also define $\mathcal{J}'=\{j^N_1,\ldots,j^N_{L-1}\}$, and $\mathcal{J}''=\{j^N_2, \ldots,j^N_{L-2}\}$.

If our Brownian motion has some stopping rule $\tau$, then the discrete process also has some rule $\tilde{\tau}$, defined to be the time $t$ such that $\tau^N_{t-1}<\tau\leq \tau^N_{t}$. We consider the probabilities
\begin{align*}
\pjt^N&=\P\left(Y_t^N=x^N_j, \medspace \tilde{\tau}\geq t+1 \right)\\
\qjt^N&=\P\left(Y_t^N=x^N_j, \medspace \tilde{\tau}=t\right).
\end{align*}
If our continuous process starts at $x^N_{j^*}$, for some $j^*$, then we have that $\psub{j^*,0}=1$, $\psub{j,0}=0$ for any $j\in\mathcal{J}\setminus\{j^*\}$, and $\qsub{j,0}=0$ for all $j$. We also consider absorbing upper and lower barriers, so $p_{j^N_0,t},p_{j^N_L,t}=0$ for all $t$.

If there exists a maximiser $\tau$ of \eqref{OptSEP}, we can discretise a Brownian motion with this stopping rule to create a random walk as above. This random walk will be a martingale and will embed some distribution $\mu^N$, so the $p,q$ associated to this stopping rule will be feasible solutions to the following problem:
\begin{align*}
\mathcal{P}': \medspace \sup_{p,q} & \sum_{\substack{j\in\mathcal{J} \\ t\geq 1}} \bar{F}_{j,t}^N \qjt  \quad \text{over } (\pjt)_{\substack{j\in\mathcal{J}' \\ t\geq 1}} ,(\qjt)_{\substack{j\in\mathcal{J} \\ t\geq 1}}&& \\
\text{subject to  }& \bullet \medspace p_{j,t},q_{j,t}\geq 0, \quad \forall j,t \\
& \bullet \medspace \sum_{t=1}^\infty \qjt = \mu^N(\{x_j^N\}), \quad &&\forall j\in\mathcal{J}\\
& \bullet \medspace \pjt + \qjt = \frac{1}{2}(p_{j-1,t-1} + p_{j+1,t-1}), \quad &&\forall t\geq 2, j\in\mathcal{J}'' \\
& \bullet \medspace p_{j^N_1,t}+q_{j^N_1,t}=\frac{1}{2} p_{j^N_2,t-1}, \quad &&\forall t\geq 2 \\
& \bullet \medspace p_{j^N_{L-1},t}+q_{j^N_{L-1},t}=\frac{1}{2} p_{j^N_{L-2},t-1}, \quad &&\forall t\geq 2 \\
& \bullet \medspace p_{j^*+1,1}+q_{j^*+1,1}=\frac{1}{2}, \quad p_{j^*-1,1}+q_{j^*-1,1}=\frac{1}{2} \\
& \bullet \medspace p_{j,1}+q_{j,1}=0, \quad &&\forall j\neq j^*,j^N_0,j^N_L \\
& \bullet \medspace q_{j^N_0,1}=0=q_{j^N_L,1} \\
& \bullet \medspace q_{j^N_L,t}=\frac{1}{2} p_{j^N_{L-1},t-1}, \quad &&\forall t\geq 2 \\
& \bullet \medspace q_{j^N_0,t}=\frac{1}{2} p_{j^N_1,t-1}, \quad &&\forall t\geq 2.
\end{align*}
If $F$ is the function in \eqref{OptSEP} we are trying to maximise, then $\bar{F}$ will be a discrete version of $F$ chosen so that $\bar{F}^N\left(\lfloor \sqrt{N}x \rfloor, \lfloor Nt \rfloor\right)\rightarrow F(x,t)$.

Note that for two different paths $\omega,\hat{\omega}$, we could have for example $\tau^N_3(\omega) < \tau^N_2(\hat{\omega})$, which means that in physical time, $q_{j,3}$ could be stopping mass before $q_{j,2}$. We will return to this limiting behaviour later, but for know we think instead of $\mathcal{P}'$ as describing the dynamics of a random walk on a fixed grid $(x^N_j,t^N_n)$. The choice of our grid should be such that we can regain a Brownian motion in the limit as $N\rightarrow\infty$, and since $\Ep{\tau^N_k-\tau^N_{k-1}\big| \F_{\tau^N_{k-1}}} = \frac{1}{N}$, we can quickly verify that $t^N_n=\frac{n}{N}$ is the correct time-step choice. 

The aim of this discretisation is to allow us to appeal to primal-dual results in linear programming theory to learn something about the properties of the continuous time primal solution, and also our continuous time  \textquoteleft dual', the superhedging problem. The primal problem $\mathcal{P}'$ is an infinite problem, and so standard strong duality results do not apply. One option is to cut off our problem at some finite time $T$, use linear programming theory on this finite problem, and then recover our infinite time problem through letting $T\rightarrow \infty$. To avoid this extra limiting argument, we keep the infinite time scale, and instead work with a modified version of $\mathcal{P}'$ that will give an equivalent optimal value, but allows us to use results from infinite-dimensional programming. 

These results rely largely on the existence of interior feasible points, and so our new problem must have inequality constraints. To allow this we drop the $q$ variables from the formulation and instead just think of $q_{j,t}=\half \left(p_{j-1,t-1}+p_{j+1,t-1}\right)-p_{j,t}$. We also change the embedding condition to a potential function constraint, requiring that the potential function of the terminal distribution of our process lies above that of $\mu^N$. For Brownian motion we know that $\Ep{L_\tau^x(B)}= \Ep{\left| B_\tau - x \right|} - \Ep{ \left| B_0 - x \right|}= -U_\mu(x) + U_0(x)$ if $B_\tau\sim \mu$ and $(B_{t\wedge\tau})_t$ is uniformly integrable. In particular, if $\tau$ embeds a distribution with potential function greater than that of $\mu$, then $\Ep{L_\tau^x(B)}\leq -U_\mu(x) + U_0(x)$. The discrete time analogue of the local time accrued at $x^N_j$ is, in this case, $\sum_{t} p_{j,t}$, and so this condition becomes $\sum_{t=0}^\infty \pjt \leq U_j$, where $U_j$ is defined below. We will assume that $U_j>0$ for every $j$, and we also know that each $U_j$ is finite, which tells us immediately that $(p_{j,t})\in l^1$ when we consider it as an infinite sequence. For the conditions of strong duality we need to work with a smaller set than $l^1$, but we later show that our smaller space forms a set of stopping times which are dense in the set of stopping times from the $l^1$ problem, and we can recover the result in $l^1$. The idea is that we restrict ourselves to $(\pjt)\in l^1(\boldsymbol{\pi}^{-1}):=\left\{(\pjt)_{j,t}: \medspace \sum_{j,t} \left| \pjt \pi_{j,t}^{-1} \right|<\infty \right\}$, where $\pi_{j,t}$ are the values of $p$ we get from running the random walk and stopping only when we hit the boundaries $j=j^N_0,j^N_L$. It will be useful to know more about these probabilities $\pi_{j,t}$ so that we know how the $p_{j,t}$ must decay.

Note that there will be many points at which we must have $\pi=0$ since the random walk simply cannot visit them, for example $\pi_{j,t}=0$ whenever $|j-j^*|>t$, and also we can only visit every other point at each $x$-level, depending on whether the time is odd or even. To make sense of these in our problem we simply do not include $\pi_{j,t}$ as a variable for these points.

\begin{lemma} \thlabel{pidecay}
The sequence $(\pi_{j,t})$ of probabilities $\pi_{j,t}=\P\left(Y_t^N=x^N_j, \medspace t<H_{x^*}\wedge H_{x_*}\right)$ is in $l^1$, and in particular, there is a vector $(m_j)_j$ such that for each $j$,
\begin{equation*}
\frac{\pi_{j,t}}{\rho^t}\rightarrow m_j \quad \text{ as }t\rightarrow\infty.
\end{equation*}
\end{lemma}
\begin{proof}
We can of course argue that $(\pi_{j,t})\in l^1$ since $\sum_{t\geq 0} \pi_{j,t} = U^{\pi}_j$ for $U^{\pi}$ the appropriate potential function, but we give a proof which allows us to deduce something of the form of the $\pi_{j,t}$.

Consider running the random walk with the above absorbing region up until some time $t$, where we have some distribution of the remaining mass. Paths at time $t$ leaving the centre-most point, call it $\hat{j}$, take longest to be absorbed at the barriers, but all mass leaving this point will be absorbed in a finite time almost surely. In particular, if we fix a small $\eps>0$, then there is some large $M$ such that
\begin{align*}
\P\left(\text{path leaving } (\hat{j},t) \text { will be absorbed by } t+M\right)&\geq \eps, \\
\text{or } \quad \P\left(\text{path leaving } (\hat{j},t) \text { hasn't been absorbed by } t+M\right)&\leq 1-\eps.
\end{align*}
By our choice of $\hat{j}$, for any $s\geq t+M$ we have $\sum_j \pi_{j,s} \leq (1-\eps) \sum_j \pi_{j,t}$. If we take $t=0$, then we know that $\sum_j \pi_{j,0}=1$ and $\sum_{j,0\leq r < M} \pi_{j,r} \leq (L+1) M$, so by the above reasoning,
\begin{align*}
\sum_{j,t} \pi_{j,t} &= \sum_{\substack{j \\ 0\leq r < M}} \pi_{j,r} + \sum_{\substack{j \\ M\leq r < 2M}} \pi_{j,r} +\sum_{\substack{j \\ 2M\leq r < 3M }} \pi_{j,r} +\cdots \\
&\leq (L+1) M + (L+1) M (1-\eps) + (L+1) M (1-\eps)^2 + \cdots \\
&= (L+1)M \sum_{n=0}^{\infty} (1-\eps)^n \\
&= \frac{(L+1)M}{\eps} < \infty.
\end{align*}

We now have that $(\pi_{j,t})\in l^1$, and can see that the sequence appears to have some approximate exponential decay.

We have a discrete-time Markov chain on $S\cup \{0\}$, where $0$ is our absorbing state (the barriers), and $S:=\{1,\dots,L-1\}$ is an irreducible set of transient states. We can easily find the leading eigenvalue and corresponding eigenvector of our transition matrix, and this gives us a quasi-stationary distribution for the process. From results on Yaglom limits of periodic Markov chains, see for example Theorem 9 of \cite{Ferrari:2015aa}, the q.s.d. has a Yaglom limit. We also know, from standard results, that the survival probability of the process decays exponentially, like $\rho^t$ for $0<\rho<1$ the leading eigenvalue. Combining these results shows that $\frac{\pi_{j,t}}{\rho^t}\rightarrow m_j$, for $m_j=\sqrt{\frac{2}{L}}\sin\left(\frac{j\pi}{L}\right)>0$ and $0<\rho=\cos\left(\frac{\pi}{L}\right)<1$. 

\end{proof}

With this result in mind, and for technical reasons, we restrict our $(p_{j,t})$ to \break $l^1(\lambda):=\left\{(\pjt)_{j,t} : \medspace \sum_j \left| \pjt \right| \lambda^t < \infty \right\}$, for some constant $\lambda>\rho^{-1}>1$. Note here that we are defining $l^1(\lambda)$ in such a way that for a fixed $t$, $p_{j,t}$ is multiplied by $\lambda^t$ for all $j$. Recall that we are optimising over probabilities $(p_{j,t})_{j,t}$ for $j\in\mathcal{J}'$ and $t\geq1$ our discrete time steps, each of length $\frac{1}{N}$ in continuous time. We choose a $j^{*,N}\in\mathcal{J}$ to start our random walk at such that $x_{j^*}^N=\frac{\lfloor j^{*,N}\rfloor}{N}\rightarrow0$ as $N\rightarrow\infty$. Our primal problem is then as follows:

\begin{align*}
\mathcal{P}^N(\lambda): \medspace \sup_{p} & \Bigg\{ \sum_{\substack{j\in\mathcal{J}'' \\ t\geq 2}} \bar{F}_{j,t}^N \left(\half\left(p_{j-1,t-1}+p_{j+1,t-1}\right)-p_{j,t}\right) + \sum_{t\geq2} \bar{F}_{j^N_L,t}^N \half p_{j^N_{L-1},t-1} +  \sum_{t\geq2} \bar{F}_{j^N_0,t}^N \half p_{j^N_1,t-1}  \\
& \qquad \qquad + \sum_{t\geq2} \bar{F}_{j^N_{L-1},t}^N\left( \half p_{j^N_{L-2},t-1} - p_{j^N_{L-1},t}\right) + \sum_{t\geq2} \bar{F}_{j^N_1,t}^N\left( \half p_{j^N_2,t-1} - p_{j^N_1,t}\right)  \\
& \qquad \qquad \qquad \qquad \qquad \qquad \qquad +  \bar{F}_{j^*+1,1}^N \left(\half - p_{j^*+1,1}\right) + \bar{F}_{j^*-1,1}^N \left(\half - p_{j^*-1,1}\right) \Bigg\},  \\
\end{align*}
over $(\pjt)_{\substack{j\in\mathcal{J}' \\ t\geq 1}}$ subject to
\begin{align*}
& \bullet \medspace (p_{j,t})\in l^1(\lambda) && \\
& \bullet \medspace p_{j,t}\geq 0, \quad \forall j,t && \\
& \bullet \medspace \mathbbm{1}\{j=j^*\}+\sum_{t=1}^\infty \pjt \leq \sqrt{N}\left(\sum_i |x^N_i-x^N_j| \mu^N(\{x^N_i\}) - |x^N_{j^*}-x^N_j|\right)=:U^N_j, &&\forall j\in\mathcal{J}' \\
& \bullet \medspace \pjt \leq \frac{1}{2}(p_{j-1,t-1} + p_{j+1,t-1}), \quad &&\forall t\geq 2, j\in\mathcal{J}'' \\
& \bullet \medspace p_{j^N_1,t}\leq\frac{1}{2} p_{j^N_2,t-1}, \quad &&\forall t\geq 2 \\
& \bullet \medspace p_{j^N_{L-1},t}\leq\frac{1}{2} p_{j^N_{L-2},t-1}, \quad &&\forall t\geq2  \\
& \bullet \medspace p_{j^*+1,1}\leq\frac{1}{2}, \quad p_{j^*-1,1}\leq\frac{1}{2} \\
& \bullet \medspace p_{j,1}=0, \quad &&\forall j\neq j^*\pm1. \\
\end{align*}
We leave the conditions at $t=1$ as equalities for clarity, but as with the $\pi_{j,t}$ there will be many points we do not visit at which we must have $p=q=0$. We do not include $p_{j,t}$ as a variable for these points. Denote by $\mathcal{P}^N$ the above problem without the restriction of $(p_{j,t})\in l^1(\lambda)$, i.e. we just require $p\in l^1$, so $\mathcal{P}^N=\mathcal{P}(1)^N$, and let $\mathrm{P}^N(\lambda)$, $\mathrm{P}^N=\mathrm{P}^N(1)$ be the optimal values of these two problems. We will show later that $\mathrm{P}^N=\mathrm{P}^N(\lambda)$. We will also introduce dual problems which we will denote by $\mathcal{D}^N(\lambda),\mathcal{D}^N$ with optimal values $\mathrm{D}^N(\lambda),\mathrm{D}^N$.

\begin{remark}
Note that our problem is of the form
\begin{align*}
\sup_{p} \Phi(p) \quad &\text{over } (p)\in l^1\left(\lambda\right) \\
\text{subject to}& \\
&\bullet \medspace p_{j,t}\geq 0 \quad \forall(j,t) \\
&\bullet \medspace Ap\geq B,
\end{align*}
where $\Phi$ is linear, and $A,B$ are given by
\begin{equation}\label{AB}
A=
\begin{pmatrix}
-\sum_{t=1}^\infty \pjt \\
-p_{j^*\pm1,1} \\
\frac{1}{2}(p_{j-1,t-1} + p_{j+1,t-1})-\pjt \\
\half p_{2,t-1}-p_{1,t} \\
\vdots \\
\end{pmatrix}, \qquad
B=
\begin{pmatrix}
-U_j \\
-\half \\
0 \\
0 \\
\vdots \\
\end{pmatrix}.
\end{equation}

\end{remark}

\section{Duality}
\subsection{Strong Fenchel Duality}
With our choice of primal problem $\mathcal{P}$, we construct a dual problem and show a strong Fenchel duality result using the following theorem from \cite[Theorem~4.4.3]{Borwein:2006ab}.
\begin{theorem} \thlabel{SFD}
Let $\mathsf{X}$ and $\mathsf{Y}$ be Banach spaces, let $f:\mathsf{X}\rightarrow \R\cup\{\infty\}$ and $g:\mathsf{Y}\rightarrow \R\cup\{\infty\}$ be convex functions and let $A:\mathsf{X}\rightarrow \mathsf{Y}$ be a bounded linear map. Define the primal and dual values $\mathsf{p},\mathsf{d}\in[-\infty,\infty]$ by the Fenchel problems
\begin{align*}
\mathsf{p}&=\inf_{x\in \mathsf{X}} \left\{ f(x) + g\left(Ax\right) \right\} \\
\mathsf{d}&=\sup_{y^*\in \mathsf{Y}^*}\left\{ -f^*\left(A^*y^*\right)-g^*\left(-y^*\right)\right\}.
\end{align*}
Then $\mathsf{p}=\mathsf{d}$, and the supremum in the dual problem is achieved if either of the following hold
\renewcommand{\labelenumi}{(\roman{enumi})}
\begin{enumerate}
\item $0\in \core\left(\dom(g)-A\dom(f)\right)$ and $f,g$ are lower semi-continuous 
\item $A\dom(f)\cap\cont(g)\neq\varnothing$.
\end{enumerate}
\end{theorem}

We briefly explain the notation in the theorem, and then use it to prove duality. For a functional $f:\mathsf{X}\rightarrow \R\cup\{\infty\}$, its convex conjugate is the function $f^*:\mathsf{X}^*\rightarrow \R\cup\{\infty\}$ given by $f^*(x^*)=\sup_{x\in \mathsf{X}} \left\{ \langle x^*,x \rangle - f(x) \right\}$. Recall also that $\dom(f)=\{x: \medspace f(x)<\infty \}$, $\cont(g)=\{y:\medspace g \text{ is continuous at } y\}$, and for $S\subseteq \mathsf{Y}$, $s\in\core(S)$ if $\cup_{\alpha>0} \alpha(S-s)=\mathsf{Y}$.

As mentioned earlier, we consider $\mathsf{X}=l^1\left(\lambda\right)$ and we think of $\mathsf{Y}$ as being $\R^{L+1}\times l^1(\lambda)$, where the first $L+1$ variables correspond to the $U_j$ conditions. Elements of $\mathsf{Y}^*=\R^{L+1}\times l^{\infty}\left(\lambda^{-1}\right)$ will be written as $y^*=(\nu_j,\eta_{j,t})$. We take $f,g$ to be
\begin{align*}
f(p)&:=
\begin{cases}
\Phi(p)=\sum_{j,t}\left(-\bar{F}^N_{j,t}\right)\left(\half\left(p_{j-1,t-1}+p_{j+1,t-1}\right)-p_{j,t}\right), & p\geq0 \\
\infty, & \text{otherwise}
\end{cases} \\
g(y)&:=
\begin{cases}
0, & y\geq B \\
\infty, & \text{otherwise},
\end{cases}
\end{align*}
and $\Phi,A,B$ are as in \eqref{AB}. With these functions it is clear that the optimal value $\mathsf{p}$ is then $-\mathrm{P}^N(\lambda)$. Also, $f,g$ are convex and lower semi-continuous, and $A$ is a linear, bounded operator. The corresponding conjugates are
\begin{align*}
f^*(x^*)&=
\begin{cases}
\infty, & x^*_{j,t}>\bar{F}^N_{j,t}-\half\left(\bar{F}^N_{j+1,t+1}+\bar{F}^N_{j-1,t+1}\right) \text{ for any } (j,t) \\
\half\left(\bar{F}^N_{j^*+1,1}+\bar{F}^N_{j^*-1,1}\right), & \text{otherwise}
\end{cases} \\
g^*(y^*)&=
\begin{cases}
\infty, & \nu^*_j>0 \text{ or } \eta^*_{j,t}>0 \text{ for any } j,(j,t) \\
\sum_{j=1}^{L-1} \nu^*_j \left(-U_j\right) - \half\left(\eta_{j^*+1,1}+\eta_{j^*-1,1}\right), & \text{otherwise}.
\end{cases}
\end{align*}

Before showing that these functions satisfy condition $(i)$ of the theorem, we find the dual problem. The dual operator $A^*$ is the functional $A^*:\mathsf{Y}^*\rightarrow \mathsf{X}^*$ satisfying $\langle A^*y^*,x\rangle=\langle y^*,Ax \rangle$ $\forall x\in \mathsf{X}$. Now,
\begin{align*}
\langle y^*,Ax \rangle&= \sum_{j\in\mathcal{J}'} \nu_j\left(-\sum_{t\geq 1} p_{j,t} \right) +\sum_{\mathclap{\substack{1\leq i \leq L-1 \\ t\geq 2}}} \eta_{j^N_i,t} \left(\frac{1}{2}( p_{j^N_{i-1},t-1} \mathbbm{1}\{i\geq2\}+ p_{j^N_{i+1},t-1} \mathbbm{1}\{i\leq L-2\}) -p_{j^N_i,t} \right) \\
&\qquad \qquad \qquad \qquad \qquad \qquad \qquad \qquad \qquad \qquad +\eta_{j^*+1,1} \left(-p_{j^*+1,1}\right) + \eta_{j^*-1,1} \left(-p_{j^*-1,1}\right) \\
&= \sum_{\mathclap{\substack{j\in\mathcal{J}' \\ t\geq 1}}} p_{j,t} \left( \half \left(\eta_{j+1,t+1}+\eta_{j-1,t+1}\right)-\eta_{j,t}-\nu_{j}\right),
\end{align*}
and so, for $y^*\in \mathsf{Y}^*$,
\begin{equation*}
A^*y^*=\left(\half \left(\eta_{j+1,t+1}+\eta_{j-1,t+1}\right)-\eta_{j,t}-\nu_{j} \right)_{\substack{j\in\mathcal{J}' \\ t\geq 1}}.
\end{equation*}

From this, we see that our Fenchel dual is
\begin{alignat}{2}
\notag \mathsf{d}=\sup \Bigg\{ -\sum_{j\in\mathcal{J}'}& \nu_{j} U_{j} - \half\left(\eta_{j^*+1,1}+\eta_{j^*-1,1}\right) -\half\left(\bar{F}^N_{j^*+1,1}+\bar{F}^N_{j^*-1,1}\right) \Bigg\}&&\\
\notag &\text{over } (\nu_j)_{j\in\mathcal{J}'}, (\eta_{j,t})_{\substack{j\in\mathcal{J} \\ t\geq 1}}, \quad \text{where } (\nu,\eta)\in l^{\infty}(\lambda^{-1}) \\ 
\label{Fdiscretesuper} \text{subject to     } & \bullet \medspace \eta_{j,t},\nu_j\geq 0, \quad &&\forall j,t \\
\label{Fetasuper} & \bullet \medspace \half\left(\eta_{j+1,t+1}+\eta_{j-1,t+1}\right)-\eta_{j,t}-\nu_j\leq \bar{F}^N_{j,t}-\half\left(\bar{F}^N_{j+1,t+1}+\bar{F}^N_{j-1,t+1}\right), \quad && \forall j,t.
\end{alignat}

All that remains to show is that we satisfy one of the two remaining conditions of the theorem.
\begin{theorem}
With $A,f,g$ defined as above, $(i)$ of \thref{SFD} holds. In particular, there is no duality gap, and the optimiser is attained in the dual problem.
\end{theorem}

\begin{proof}
A typical element of $\dom(g)-A\dom(f)$ looks like
\begin{equation*}
\begin{array}{cc}
w=\begin{pmatrix}
y_j + \sum_{t\geq 1} \pjt \\
y_{j^*\pm1,1}+p_{j^*\pm1,1} \\
y_{j,t} + \pjt -\half(p_{j-1,t-1}+p_{j+1,t-1}) \\
y_{j^N_1,t} + p_{j^N_1,t}-\half p_{j^N_2,t-1} \\
\vdots 
\end{pmatrix} &
\begin{array}{rll}
& (p_{j,t}),(y_j,y_{j,t})\in l^1(\lambda) & \\
 & \pjt\geq 0 \quad &\forall j\in\mathcal{J}', \thinspace t\geq1\\
\text{ for }& y_j\geq -U_j \quad &\forall j\in\mathcal{J}'\\
& y_{j^*\pm1,1}\geq -\half &\\
& y_{j,t}\geq 0 \quad &\forall j\in\mathcal{J}', \thinspace t\geq2.
\end{array}
\end{array}
\end{equation*}

We are required to show that for any $z\in \R^{L+1}\times l^1(\lambda) \medspace \exists \alpha>0$ and a $w$ of the above form such that $\alpha w =z$. Take $z=(z_j,z_{j,t})\in \R^{L+1}\times l^1(\lambda)$, consider a constant, $\gamma$, and let
\begin{equation*}
\rho_{j,t}^{k,s}= \gamma \left|z_{k,s}\right|\iden\{s>t\} \frac{\pi_{j,t}}{\pi_{k,s}},
\end{equation*}
for $j,k\in\mathcal{J}$ and $s,t\geq 1$, so that
\begin{equation*}
\rho_{j,t}^{k,s}-\half\left(\rho_{j-1,t-1}^{k,s}+ \rho_{j+1,t-1}^{k,s}\right)=
\begin{cases} 0, & s\neq t \\ 
-\gamma \left|z_{k,t}\right|\frac{\pi_{j,t}}{\pi_{k,t}},& s=t
\end{cases} \quad \text{for }j\in\mathcal{J}',
\end{equation*}
and similarly for the other terms.
Then $\rho_{j,t}=\sum_{k,s} \rho_{j,t}^{k,s}\geq0$ defines a set of probabilities such that,
\begin{align}
\notag \rho_{j,t} -\half(\rho_{j-1,t-1}+\rho_{j+1,t-1})=- \gamma \sum_k |z_{k,t}| \frac{\pi_{j,t}}{\pi_{k,t}}\leq -\gamma \left| z_{j,t}\right|,\\ 
\sum_t \rho_{j,t} =\gamma \sum_{k,s,t} \iden\{s>t\}|z_{k,s}|\frac{\pi_{j,t}}{\pi_{k,s}}=\gamma \sum_{k,s}|z_{k,s}|\frac{\pi_{j,0}+\ldots + \pi_{j,s-1}}{\pi_{k,s}}\in (0,\infty), \label{sumrho}
\end{align}
since $\pi\in l^1$. Also, for any $\eps>0$ there is a $T$ such that for every $j$ and $t\geq T$, $\left| \frac{\pi_{j,t}}{\rho^t}-m_j \right|<\eps$. Then
\begin{align*}
\sum_{j,t} \rho_{j,t} \lambda^t &= \gamma \sum_{j,t,k,s} \iden\{s>t\} |z_{k,s}|\lambda^t\frac{\pi_{j,t}}{\pi_{k,s}} \\
&= \gamma \sum_{\substack{k,s\leq T \\ j,t}} \iden\{s>t\} |z_{k,s}|\lambda^t\frac{\pi_{j,t}}{\pi_{k,s}} + \gamma \sum_{j,t\leq T} \pi_{j,t}\lambda^t \sum_{k,s>T} \frac{|z_{k,s}|}{\pi_{k,s}} + \sum_{k,s>T} \frac{|z_{k,s}|}{\pi_{k,s}} \sum_{j,T<t<s} \frac{\pi_{j,t}}{\rho^t}(\lambda\rho)^t \\
&<\infty.
\end{align*}
The first sum is a finite sum of finite terms, so is finite, and the second is finite since $(z_{j,t})\in l^1(\lambda)\subset l^1(\boldsymbol{\pi})$. Let $\overline{m}=\max_j m_j$ and $\underline{m}=\min_j m_j$. Then for the final sum,
\begin{align*}
\sum_{k,s>T} \frac{|z_{k,s}|}{\pi_{k,s}} \sum_{j,T<t<s} \frac{\pi_{j,t}}{\rho^t}(\lambda\rho)^t &\leq \sum_{k,s>T} \frac{|z_{k,s}|}{\pi_{k,s}} \sum_{j,T<t<s} (\lambda\rho)^t (m_j+\eps) \\
&\leq (\overline{m}+\eps)\sum_{k,s>T} \frac{|z_{k,s}|}{\pi_{k,s}} \frac{(\lambda\rho)^s-1}{\lambda\rho-1} \\
&\leq \frac{\overline{m}+\eps}{\lambda\rho-1} \sum_{k,s>T} \frac{|z_{k,s}|\lambda^s}{m_k-\eps} - \frac{\overline{m}+\eps}{\lambda\rho-1} \sum_{k,s>T} \frac{|z_{k,s}|}{\pi_{k,s}} \\
&\leq \frac{\overline{m}+\eps}{\underline{m}-\eps} \frac{1}{\lambda\rho-1} \sum_{k,s>T} |z_{k,s}|\lambda^s -\frac{\overline{m}+\eps}{\lambda\rho-1} \sum_{k,s>T} \frac{|z_{k,s}|}{\pi_{k,s}} \\
&<\infty,
\end{align*}
again since $(z_{j,t})\in l^1(\lambda)\subset l^1(\boldsymbol{\pi})$, and this also has a finite limit as $\eps\rightarrow0$.

In particular, we have $(\rho_{j,t})\in l^1\left( \lambda \right)$. For each $j,t$ we need $y_j\geq -U_j$, $y_{j,t}\geq0$, $\gamma>0$ and $\alpha>0$ such that
\begin{align*}
z_j&=\alpha \left(y_j + \sum_{t\geq 1} \rho_{j,t} \right)=\alpha \left(y_j +\gamma \sum_{k,s,t} \iden\{s>t\}|z_{k,s}|\frac{\pi_{j,t}}{\pi_{k,s}}\right), \\
z_{j,t}&=\alpha \left( y_{j,t}+\rho_{j,t} -\half(\rho_{j-1,t-1}+\rho_{j+1,t-1})\right)=\alpha \left( y_{j,t} -\gamma \pi_{j,t} \sum_k \frac{|z_{k,t}|}{\pi_{k,t}} \right).
\end{align*}
Setting $\alpha=\gamma^{-1}$, we find that
\begin{align*}
y_j&=\gamma z_j-\gamma \sum_{k,s,t} \iden\{s>t\}|z_{k,s}|\frac{\pi_{j,t}}{\pi_{k,s}}>\half \gamma z_j -  \gamma \sum_{t} \pi_{j,t} \sum_{k,s} \frac{|z_{k,s}|}{\pi_{k,s}} \geq -U_j, \quad &&\forall j\in\mathcal{J}' \\
y_{j,t}&=\gamma z_{j,t}+ \gamma \pi_{j,t} \sum_k \frac{|z_{k,t}|}{\pi_{k,t}} \geq 0, \quad &&\forall j\in\mathcal{J}', t\geq1,
\end{align*}
for any choice of $0<\gamma \leq \min_j \left\{ \left|2U_j\left(z_j- 2\sum_{t} \pi_{j,t} \sum_{k,s} \frac{|z_{k,s}|}{\pi_{k,s}}\right)^{-1}  \right| \right\}$. Now we can see that the scaling factor was necessary to ensure that even if $y_j$ is negative, we can scale it so that it is still larger than $-U_j$. Note here that division by zero isn't an issue since $\sum_{t} \pi_{j,t} \sum_{k,s} \frac{|z_{k,s}|}{\pi_{k,s}}$ is a strict upper bound of $\sum_t \rho_{j,t}$ by \eqref{sumrho} and we can therefore choose a smaller bound if $z_j=2\sum_{t} \pi_{j,t} \sum_{k,s} \frac{|z_{k,s}|}{\pi_{k,s}}$. To see that $(y_{j,t})\in l^1(\lambda)$ note that $(z_{j,t}),(\rho_{j,t})\in l^1(\lambda)$ and $y_{j,t}=\gamma z_{j,t} + \half(\rho_{j-1,t-1}+\rho_{j+1,t-1})-\rho_{j,t}$.
\end{proof}

\subsection{Interpretation of Dual Variables} \label{dualsection}
In this section we outline the interpretation of the dual problem and its relation to the continuous time problem of minimising the cost of superhedging strategies. This work was motivated by the problem of finding a robust upper bound on the price of a European option on a leveraged exchange traded fund, and in particular in \cite{Cox:2016aa} we use results from this paper to give the form of the optimal superhedging strategy in that case. Here we give a brief overview of how the continuous time hedging strategy can be obtained from $(\nu^*,\eta^*)$, the optimisers of $\mathcal{D}^N(\lambda)$.

Our Fenchel dual problem is
\begin{flalign}
\notag &\mathcal{D}^N(\lambda): \quad \text{ minimise } \Bigg\{ \sum_{j\in\mathcal{J}'} \nu_j U_j + \half\left(\eta_{j^*+1,1}+\eta_{j^*-1,1}\right) +\half\left(\bar{F}^N_{j^*+1,1}+\bar{F}^N_{j^*-1,1}\right) \Bigg\}&\\
\notag &\text{over } (\nu_j)_{j\in\mathcal{J}'}, (\eta_{j,t})_{\substack{j\in\mathcal{J} \\ t\geq 1}} \text{ subject to } &\\
&\bullet (\nu,\eta)\in l^{\infty}(\lambda^{-1}) &\\ 
\label{Fdiscretesuper2} &\bullet \medspace \eta_{j,t},\nu_j\geq 0, \quad &&\forall j,t &\\
\label{Fetasuper2}  &\bullet \medspace \half\left(\eta_{j+1,t+1}+\eta_{j-1,t+1}\right)-\eta_{j,t}-\nu_j\leq \bar{F}^N_{j,t}-\half\left(\bar{F}^N_{j+1,t+1}+\bar{F}^N_{j-1,t+1}\right), \quad && \forall j,t. &
\end{flalign}

In Lagrangian duality we know that we have duality exactly when the complementary slackness conditions hold, and this is also true here. Fix $N$ and take the optimal dual solution $y^*=(\nu^*,\eta^*)$ and any primal feasible $p$, so $g(Ap)=0$. Then,
\begin{align*}
\mathsf{d}&=-\sum_{j\in\mathcal{J}'} \nu^*_j U_j - \half\left(\eta^*_{j^*+1,1}+\eta^*_{j^*-1,1}\right) -\half\left(\bar{F}^N_{j^*+1,1}+\bar{F}^N_{j^*-1,1}\right) \\
&\leq -\sum_{j\in\mathcal{J}'} \nu^*_j U_j - \half\left(\eta^*_{j^*+1,1}+\eta^*_{j^*-1,1}\right) -\half\left(\bar{F}^N_{j^*+1,1}+\bar{F}^N_{j^*-1,1}\right) \\
& \qquad \qquad \qquad \qquad \qquad \qquad \qquad \qquad \qquad + \sum_{\substack{j\in\mathcal{J}' \\ t\geq2}} \eta^*_{j,t}\left(\half\left(p_{j+1,t-1}+p_{j-1,t-1}\right)-p_{j,t}\right) \\
&\leq \sum_{\substack{j\in\mathcal{J}' \\ t\geq1}} p_{j,t}\left(\half\left(\eta^*_{j+1,t+1}+\eta^*_{j-1,t+1}\right)-\eta^*_{j,t}-\nu^*_j\right) -\half\left(\bar{F}^N_{j^*+1,1}+\bar{F}^N_{j^*-1,1}\right) \\
&\leq \sum_{\substack{j\in\mathcal{J}' \\ t\geq1}} p_{j,t}\left(\bar{F}^N_{j,t}-\half\left(\bar{F}^N_{j+1,t+1}+\bar{F}^N_{j-1,t+1}\right)\right) -\half\left(\bar{F}^N_{j^*+1,1}+\bar{F}^N_{j^*-1,1}\right) \\
&=f(p)+g(Ap).
\end{align*}
We have equality in the above inequalities, and $p$ is a primal optimiser, if and only if the following complementary slackness conditions hold:
\begin{align}
\label{FCS1} \pjt>0 & \implies  \frac{1}{2} \left( \eta_{j-1,t+1} + \eta_{j+1,t+1} \right)-\eta_{j,t}-\nu_j=\bar{F}^N_{j,t}-\half\left(\bar{F}^N_{j+1,t+1}+\bar{F}^N_{j-1,t+1}\right) \\
\label{FCS2} q_{j,t}>0 &\implies  \eta_{j,t}=0 \\
\label{FCS3} \nu_j>0 & \implies \sum_{t=1}^{\infty} \pjt = U_j,
\end{align}
where as always we define $q_{j,t}=\half\left(p_{j-1,t-1}+p_{j+1,t-1}\right)-p_{j,t}$.

It is well known that in the continuous time problem of finding an upper bound on the price of an exotic option that the dual problem corresponds to minimising the cost of a super-replicating portfolio, and much work has been done in finding conditions under which there is no duality gap. In \cite{Cox:2013ac} and \cite{Cox:2013ab} the authors give the form of the optimal superhedging strategy in the case where the optimal stopping time corresponds to a Root or Rost hitting time. They define functions $G(x,t)$, representing a dynamic trading strategy, and $H(x)$, corresponding to a fixed portfolio of options, in terms of the stopping time/stopping region. If $G,H$ are defined such that for a feasible stopping time $\sigma$ we have
\begin{list}{$\bullet$}{}
\item $F(x,t)\leq G(x,t)+H(x)$ everywhere
\item $G(W_t,t)$ is a supermartingale
\item $F(W_{\sigma},\sigma)=G(W_{\sigma},\sigma)+H(W_{\sigma})$
\item $G(X_{t \wedge \sigma},t \wedge \sigma)$ is a martingale,
\end{list}
then we have the optimal dual portfolio and no duality gap. The first two conditions above are dual feasibility conditions, and the final two ensure optimality, so we can think of them as complementary slackness conditions. We can then see how we might recover these functions from our dual problem.

We require functions $G,H$ such that $G+H$ is a superhedge, and $G(W_t,t)$ is a supermartingale, and these conditions correspond to \eqref{Fdiscretesuper} and \eqref{Fetasuper} respectively. If we define $\tilde{\eta}^*_{j,t}=\eta_{j,t}+\bar{F}^N_{j,t}$ then \eqref{Fdiscretesuper}, \eqref{Fetasuper} become
\begin{align*}
&\tilde{\eta}_{j,t}\geq \bar{F}^N_{j,t} \quad &&\forall(j,t) \\
&\tilde{\eta}_{j,t}-\half\left(\tilde{\eta}_{j+1,t+1}+\tilde{\eta}_{j-1,t-1}\right) +\nu_j\geq 0 \quad &&\forall(j,t),
\end{align*}
so $\tilde{\eta}$ has the superhedging and supermartingale properties we are looking for. It then makes sense to hope that we find 
\begin{equation*}
\lim_{N\rightarrow\infty} \eta^*_{\lfloor\sqrt{N}x\rfloor,\lfloor Nt\rfloor} = G(x,t) + H(x) -F(x,t).
\end{equation*}
Our complementary slackness conditions \eqref{FCS1} and \eqref{FCS2} then correspond to $G(W_t,t)$ being a martingale in the continuation region and having an exact hedge in the stopping region. From \eqref{FCS1} we have
\begin{align*}
\nu^*_j&= \half \left( \tilde{\eta}^*_{j+1,t+1}-2\tilde{\eta}^*_{j,t+1} + \tilde{\eta}^*_{j-1,t+1}\right) + \left(\tilde{\eta}^*_{j,t+1}-\tilde{\eta}^*_{j,t}\right)  \\
&=\half\left(\frac{\left(\tilde{\eta}^*_{j+1,t+1}-\tilde{\eta}^*_{j,t+1}\right)-\left(\tilde{\eta}^*_{j,t+1} - \tilde{\eta}^*_{j-1,t+1}\right)}{\frac{1}{N}}\right) \frac{1}{N} + N \left(\tilde{\eta}^*_{j,t+1}-\tilde{\eta}^*_{j,t}\right) \frac{1}{N},
\end{align*}
so $N\nu^*_{\lfloor\sqrt{N}x\rfloor}\rightarrow \half H''(x)$ if $G(X_{t\wedge\tau},t\wedge\tau)$ is a martingale and we have the suitable differentiability. Also, our dual objective function then becomes $\Ep{G(W_{\tau},\tau) + H(W_{\tau})}=G(W_0,0)+\Ep{H(W_{\tau})}$, the cost of the superhedging portfolio, since 
\begin{equation*}
\sum_j \nu^*_j U_j \rightarrow \int \half H''(x) \left(U_{\delta_0}(x)-U_{\mu}(x)\right) \dx = \Ep{H(W_{\tau})}-H(W_0) \quad \text{if } W_{\tau}\sim \mu.
\end{equation*}

\subsection{Strong Duality of $\mathcal{P}^N,\mathcal{D}^N$}
We now have strong duality in the $\lambda$ problems, and attainment in the dual problem. In this section we show that we have primal attainment in $\mathcal{P}^N$, the $l^1$ problem, and that the optimal value in this case agrees with the $l^1(\lambda)$ problem, $\mathrm{P}^N=\mathrm{P}^N(\lambda)$.

\begin{lemma} \thlabel{primalattain}
Choose $\bar{F}^N$ so that $\mathrm{P}^N<\infty$. Then $\mathrm{P}^N=\mathrm{P}^N(\lambda)$ and the supremum $\mathrm{P}^N$ is attained by some sequence $p^*\in l^1$.
\end{lemma}
\begin{proof}
To show primal attainment in $\mathcal{P}^N$, we show that the feasible region of $\mathcal{P}^N$ is a compact subset of $l^1$, and therefore supremums are attained. It is well known that a metric space is compact if it is totally bounded and complete. We argue that completeness of our feasible region follows from the formulation of $\mathcal{P}^N$, in particular we have no strict inequalities, so if the limit of a sequence of feasible solutions exists, then it will also be feasible. To show total boundedness we use the Kolmogorov-Riesz Compactness Theorem equivalent for $l^p$ spaces from \cite{Hanche-Olsen:2010aa}, first proved in \cite{Frechet:1908aa}.

\begin{theorem*}
A subset of $l^r$, where $1\leq r <\infty$, is totally bounded if and only if, 
\renewcommand{\labelenumi}{(\roman{enumi})}
\begin{enumerate}
\item it is pointwise bounded, and
\item for every $\eps >0$ there is some $n$ so that, for every $x$ in the given subset,
\end{enumerate}
\begin{equation*}
\sum_{k>n} \left| x_k \right|^r < \eps^r.
\end{equation*}
\end{theorem*}

It is clear that our sequences are pointwise bounded, but also note that for any sequence $(p_{j,t})$ in our feasible region, we have that $p_{j,t}\leq \pi_{j,t}$ $\forall(j,t)$, so in particular,
\begin{equation*}
\sum_{j,t>n} p_{j,t} \leq \sum_{j,t>n} \pi_{j,t}.
\end{equation*}
By \thref{pidecay}, $(\pi_{j,t})\in l^1$, and therefore $\forall \eps>0 \medspace \exists n$ such that $\sum_{j,t>n} \pi_{j,t}<\eps$, and we are done.

Our strong duality result, \thref{SFD}, proves that $\mathrm{D}^N(\lambda)=\mathrm{P}^N(\lambda)$, but we also have weak duality in the original discretised problem, so $\mathrm{D}^N\geq \mathrm{P}^N$, where $\mathrm{D}^N(\lambda)$ is the value of our dual problem taking $(\nu,\eta)\in \R^{L+1}\times l^{\infty}(\lambda^{-1})$, and $\mathrm{D}^N$ is the optimal dual value with $(\nu,\eta)\in\R^{L+1}\times l^{\infty}$. Also,
\begin{alignat*}{4}
l^1(\lambda)\subseteq l^1 &\implies  &\mathrm{P}^N&\geq \mathrm{P}^N(\lambda) \\
l^{\infty}\subseteq l^{\infty}(\lambda^{-1}) &\implies & \mathrm{D}^N&\geq \mathrm{D}^N(\lambda).
\end{alignat*}

Let $p^*$ be an optimiser of $\mathcal{P}^N$, whose existence we have just proven. Since $l^1(\lambda)\subseteq l^1$, we either have that $p^*\in l^1(\lambda)$, or we can \textquoteleft cut-off' $p^*$ at some finite time, and we get a feasible solution to $\mathcal{P}^N(\lambda)$. Define $p^T_{j,t}=p^*_{j,t}\iden\{t<T\}$, then we have an approximating sequence, $(p^T)_T$, of feasible solutions of $\mathcal{P}^N(\lambda)$, with $\sum_{j,t}\bar{F}^N_{j,t}q^T_{j,t}\rightarrow\sum_{j,t}\bar{F}^N_{j,t}q^*_{j,t}$ provided $\sum_{j,t}\bar{F}^N_{j,t}q^*_{j,t}<\infty$. We must therefore have $\mathrm{P}^N= \mathrm{P}^N(\lambda)$.
\end{proof}

\begin{remark}
If $(\nu^*,\eta^*)$ are the optimisers of $\mathcal{D}^N(\lambda)$ then, similarly to above, we can consider $\hat{\nu}_j=\nu^*_j$ and $\hat{\eta}^T_{j,t}=\begin{cases} \eta^*_{j,t}, & t<T \\ \eta^*_{j,t}\wedge\bar{F}^N_{j,t}, & t\geq T \end{cases}$. Then $(\hat{\nu},\hat{\eta}^T)$ are $\mathcal{D}^N$-feasible for large $T$ and certain $\bar{F}^N$. An important example is where $\bar{F}^N$ is the discretisation of some European call option which is decreasing in time. In this case there is some $T^*$ such that $\bar{F}^N_{j,t}=0$ for all $t\geq T$, and then we find $\mathrm{D}^N= \mathrm{D}^N(\lambda)$ and $(\hat{\nu},\hat{\eta}^T)$ attain $\mathrm{D}^N$ for any $T>T^*$. This is the case of the leveraged exchange traded fund payoff from \cite{Cox:2016aa}. 
\end{remark}

\section{The Cave Embedding Case} \label{cave}
So far we have made no assumptions on our functions $F,\bar{F}^N$ except that they give well-defined optimisation problems, and that $\bar{F}^N\rightarrow F$ in some sense. We now show that for certain choices of $F$, these discrete optimisation problems have certain properties that have already been shown in the continuous time case. In particular, we focus on the ideas of the Root, Rost, and cave embeddings, first given in \cite{Root:1969aa}, \cite{Rost:1971aa}, and \cite{Beiglboeck:2013aa} respectively. We concentrate on the cave embedding example, since it is a combination of a Root barrier and a Rost inverse-barrier, and therefore incorporates the arguments of the other two problems. Recall from \cite{Beiglboeck:2013aa} that a region $\mathcal{R} \subseteq \mathbb{R} \times \mathbb{R}_+$ is a cave barrier if there exists a $t_0 \in \R_+$, an inverse barrier $\mathcal{R}^0 \subseteq \R\times[0,t_0]$, and a barrier $\mathcal{R}^1 \subseteq \R\times[t_0,\infty)$ such that $\mathcal{R}= \mathcal{R}^0 \cup \mathcal{R}^1$. In \cite{Beiglboeck:2013aa} it is proved that there exists a cave barrier $\mathcal{R}$ such that $\tau=\inf \{t\geq 0: \thinspace (W_t,t)\in \mathcal{R}\}$ minimises $\E[\varphi(\sigma)]$ over all stopping times $\sigma$ such that $W_{\sigma}\sim\mu$, where $\varphi : \R_+ \to [0,1]$ is such that
\begin{list}{$\bullet$}{}
\item $\varphi(0)=0$, $\lim_{t\to \infty} \varphi(t)=0$, $\varphi(t_0)=1$
\item $\varphi$ is strictly concave (and therefore increasing) on $[0,t_0]$
\item $\varphi$ is strictly convex (decreasing) on $[t_0,\infty)$.
\end{list}

Fix $t_0\in \R_+$ and consider our payoff to be the negative of such a function (so that we are still maximising), i.e. $F(x,t)=f(t):=-\varphi(t)$, so $\bar{F}^N_{j,t}=\bar{f}^N(t)=-\varphi\left(\frac{t}{N}\right)$. We argue that our optimal $p_{j,t}$ (we will drop the $^*$ now since we will always be thinking about the optimal solutions) define a discretised cave barrier stopping region for the random walk, and that this stopping region embeds exactly our distribution $\mu^N$, for each $N$. If we were working with the primal problem $\mathcal{P}'$ then it would be clear that we embed $\mu^N$, however it appears that this could fail with the conditions of $\mathcal{P}$, since our potential function could sit above that of $\mu^N$, so we actually embed a different distribution. From now on we will be working solely with the primal optimisers $p^*$, and so for ease of presentation we will drop the $^*$ notation. Recall that we write $\qjt:=\half(p_{j-1,t-1}+p_{j+1,t-1})-\pjt$.

\begin{lemma} \thlabel{Rostembed}
For each $N$, if $\bar{\tau}^{N}$ is the stopping time of a random walk $Y^N$ given by the primal optimisers $p$, then $Y^N_{\bar{\tau}^N}\sim\mu^N$.
\end{lemma}

\begin{proof}

If we have equality in our potential function condition, then by uniqueness, our stopping rule must embed $\mu^N$. To show this, we change our payoff function in a way that doesn't affect our problem or previous arguments, but ensures that it is never optimal to have a strict inequality in our potential condition. We know that if $W_{\sigma}\sim\mu$, then $\Ep{\sigma}=\int x^2 \mu(\di x)$ is fixed, and so we can add a linear function of time to our payoff without affecting how the optimum solution is obtained. In particular, this means that in this case we can make our payoff increasing everywhere (at least away from zero), by considering $F(x,t)+Ct$ and therefore $\bar{f}^N(t)+C\frac{t}{N}$ for some large constant $C$. This will change our dual problem, and our interpretations of $\eta^*$ and $\nu^*$ become $G(x,t)+H(x)-F(x,t)+Ct$ and $\half H''(x) +C$, respectively, by considering $\hat{G}(x,t)=G(x,t)+Ct-Cx^2$ and $\hat{H}(x)=H(x)+Cx^2$.

Now if we release mass at some $(j,r)$, since $\bar{f}^N(t)+C\frac{t}{N}$ is increasing, it will go on to score more than its current value, meaning that we always run our process for as long as possible in order to be optimal. If we have some $x_k$ such that $\sum_{t=1}^\infty p_{k,t} < U_k$, then it is easy to see that we can find a site $(i,s)$ such that $q_{i,s}>0$ and $U_i>\sum_t p_{i,t}$ and release some mass from here. Suppose there does not exist such a site $(i,s)$. Then at any $j$ at which we embed mass we have $U_j=\sum_t p_{j,t}$. Now, for any $j\neq j^*\pm1$, $\sum_{t} p_{j,t}\leq \sum_{t} \half (p_{j-1,t-1}+p_{j+1,t+1})=\half \sum_{t} p_{j-1,t} + \half \sum_{j,t} p_{j+1,t}$ with equality if and only if no mass is embedded at $x_j^N$. Then we have that $\sum_{t} p_{j,t}$ is linear between points $j$ at which we embed mass, which we know is also case in potential functions, and therefore we must have $U_j=\sum_t p_{j,t}$ everywhere, by the convexity of $U$. Therefore, by contradiction, we have such a point $(i,s)$.

Now fix $0<\eps<\min\left\{q_{i,s},U_i-\sum_t p_{i,t} \right\}$, and define $\bar{p},\bar{q}$ by
\begin{align*}
&\bar{q}_{i,s}=q_{i,s}-\eps,  &&\bar{p}_{i,s}=p_{i,s}+\eps, \\
&\bar{q}_{i+1,s+1}=q_{i+1,s+1}+\half\eps, &&\bar{q}_{i-1,s+1}=q_{i-1,s+1}+\half\eps, \\
&\bar{p}_{j,r}=p_{j,r}, && \bar{q}_{j,r}=q_{j,r} \quad \text{otherwise}.
\end{align*}
It is easy to check that $\bar{p},\bar{q}$ are $\mathcal{P}$-feasible, and also
\begin{align*}
\sum_{j,r} (\bar{f}^N(r)+C\frac{r}{N}) \bar{q}_{j,r} &= \sum_{j,r} (\bar{f}^N(r)+C\frac{r}{N}) q_{j,r} -\eps(\bar{f}^N(s)+C\frac{s}{N})+\eps(\bar{f}^N(s+1)+C\frac{s+1}{N}) \\
&= \sum_{j,r} (\bar{f}^N(r)+C\frac{r}{N}) q_{j,r} +C\frac{\eps}{N} +\eps(\bar{f}^N(s+1)-\bar{f}^N(s)) \\
&\geq \sum_{j,r} (\bar{f}^N(r)+C\frac{r}{N}) q_{j,r},
\end{align*}
if we choose $C\geq \bar{f}^N(1)-\bar{f}^N(2)$ for all $N$. Note that for each $N$ there is a $C(N)<\infty$ such that $C(N)\geq \bar{f}^N(1)-\bar{f}^N(2)$, but also, if $\partial_t F(x,t)=\varphi'(t)$ is bounded, then $\lim_{N\rightarrow\infty} C(N)<\infty$ and we can choose $C=\sup_NC(N)=\lim_{N\rightarrow\infty} C(N)$. This shows we can always improve our payoff if we do not have equality in the potential condition, and therefore it is never optimal to have a strict inequality in this condition. In particular, by uniqueness of potentials, the optimal $p$ embed $\mu^N$ into the associated random walk. 
\end{proof}

Now that we know that we embed the correct distribution, we can actually show that we do this using an almost-deterministic stopping region that has the form of a cave barrier.
\begin{theorem} \thlabel{caveshape}
The optimal solution of the primal problem $\mathcal{P}$, where $\bar{F}_{j,t}=\bar{f}(t)=-\varphi\left(\frac{t}{N}\right)$ for $\varphi$ a cave function, is given by $p_{j,t}$ which give a stopping region for a random walk with the cave barrier-like property
\begin{align*}
\text{if } q_{i,t}>0 \medspace \text{ for some } (i,t) \text{ where }t< t_0, \text{ then } p_{i,s}=0 \medspace \forall s<t, \\
\text{if } q_{i,t}>0 \medspace \text{ for some } (i,t) \text{ where }t> t_0, \text{ then } p_{i,s}=0 \medspace \forall  s>t.
\end{align*}
\end{theorem}
\begin{proof}

First consider the inverse-barrier to the left of $t=t_0$. To show this, suppose we have a feasible solution with $q_{i,t}>0$ and $p_{i,s}>0$ for some $i$ and $s<t<t_0$. We take some $0<\eps<\min\{\shalf q_{i,t}, p_{i,s}\}$ and show that we can improve our objective function (increase the payoff), by transferring $\eps$ of the mass that currently leaves $(i,s)$ onto $(i,t)$. To move the paths we need to know how this $\eps$ of mass behaves, and the following quantities \textquoteleft track' an $\eps$ mass of particles leaving $(i,s)$:
\begin{align*}
\tilde{p}_{i,s}&=\eps, \quad \tilde{p}_{j,s}=0 \quad \forall j\neq i, \quad \tilde{q}_{j,s}=0 \quad \forall j,  \\
\tilde{p}_{j,r+1}&=p_{j,r+1} \times (\text{the $\eps$ mass from } (j+1,r) + \text{ the $\eps$ mass from } (j-1,r))\\
&= p_{j,r+1} \left( \frac{p_{j+1,r}}{p_{j+1,r}+p_{j-1,r}} \frac{\tilde{p}_{j+1,r}}{p_{j+1,r}} + \frac{p_{j-1,r}}{p_{j+1,r}+p_{j-1,r}} \frac{\tilde{p}_{j-1,r}}{p_{j-1,r}}\right)\\
&= p_{j,r+1} \frac{\tilde{p}_{j+1,r} + \tilde{p}_{j-1,r}}{p_{j+1,r}+p_{j-1,r}} \qquad \quad \forall j\neq j^N_0,j^N_L, \medspace\forall r\geq s, \\
 \tilde{q}_{j,r+1} &= q_{j,r+1} \frac{\tilde{p}_{j+1,r} + \tilde{p}_{j-1,r}}{p_{j+1,r}+p_{j-1,r}} \qquad \quad \forall j\neq j^N_0,j^N_L, \medspace\forall r\geq s,\\
 \tilde{p}_{0,r+1}&=p_{0,r+1}\frac{\tilde{p}_{1,r}}{p_{1,r}}, \quad \tilde{q}_{0,r+1}=q_{0,r+1}\frac{\tilde{p}_{1,r}}{p_{1,r}} \qquad \forall r\geq s,\\
 \tilde{p}_{L,r+1}&=p_{1,r+1}\frac{\tilde{p}_{L-1,r}}{p_{L-1,r}},\quad \tilde{q}_{L,r+1}=q_{L,r+1}\frac{\tilde{p}_{L-1,r}}{p_{L-1,r}} \qquad \forall r\geq s.
\end{align*}

Using the $\tilde{p}, \tilde{q}$, and the procedure described here, we can write down the $\bar{p},\bar{q}$ corresponding to the system after the transfer of the mass. Note that in the above, and it what follows, we define $\tilde{p}, \bar{p}$, and then $\tilde{q},\bar{q}$ follow from defining $\tilde{q}_{j,t}=\half(\tilde{p}_{j-1,t-1}+\tilde{p}_{j+1,t-1})-\tilde{p}_{j,t}$, and similarly for $\bar{q}$, but we write them out for clarity. Before time $s$ nothing changes, and so we have
\begin{equation*}
\bar{p}_{j,r}=p_{j,r}, \quad \bar{q}_{j,r} = q_{j,r} \quad \forall (j,r)\in\{(j,r): \medspace 1\leq r <s\}.
\end{equation*}

At $s$ we stop $\eps$ particles that previously left $(i,s)$:
\begin{align*}
\bar{p}_{i,s}&=p_{i,s}-\tilde{p}_{i,s}=p_{i,s}-\eps, \quad &\bar{q}_{i,s}&=q_{i,s}+\eps, \\
\bar{p}_{j,s}&=p_{j,s}, \quad &\bar{q}_{j,s}&=q_{j,s} \quad \forall j\neq i.
\end{align*}

Between times $s$ and $t$ we have lost these stopped paths, so
\begin{equation*}
\bar{p}_{j,r}=p_{j,r}-\tilde{p}_{j,r}, \quad \bar{q}_{j,r}=q_{j,r} - \tilde{q}_{j,r} \quad \forall (j,r)\in\{(j,r): \medspace s< r <t\}.
\end{equation*}

At time $t$ we release an extra $\eps$ paths that were previously stopped at $(i,t)$, giving
\begin{align*}
\bar{p}_{i,t}&=p_{i,t}-\tilde{p}_{i,t}+\eps, \quad &\bar{q}_{i,t}&=q_{i,t}-\tilde{q}_{i,t}-\eps, \\
\bar{p}_{j,t}&=p_{j,t}-\tilde{p}_{j,t}, \quad &\bar{q}_{j,t}&=q_{j,t}-\tilde{q}_{j,t} \quad \forall j\neq i. 
\end{align*}

For $r>t$ we have the $\eps$ paths from $(i,t)$, but we have now lost $\eps$ from $(i,s)$, so we have
\begin{equation*}
\bar{p}_{j,r}=p_{j,r}-\tilde{p}_{j,r}+\tilde{p}_{j,r-(t-s)}, \quad \bar{q}_{j,r}=q_{j,r}-\tilde{q}_{j,r}+\tilde{q}_{j,r-(t-s)} \quad \forall (j,r)\in\{(j,r): \medspace t<r \}.
\end{equation*}

These $\bar{p},\bar{q}$ are $\mathcal{P}^N$-feasible by \thref{rostfeasible}, and in \thref{rostoptimal} we show that these do indeed increase our payoff.

The lemmas then tell us that whenever we have a primal-feasible solution $p$ where there exists $i$ and $s<t$ such that $q_{i,t}>0$ and $p_{i,s}>0$, we can improve optimality by moving mass between these points. In particular, since we know by linear programming theory that an optimiser exists, our optimal $p$ cannot have this property, so we have a discrete form of an inverse barrier: if $q_{i,t}>0$ for some $(i,t)$ where $t<t_0$, then $p_{i,s}=0$ $\forall s<t$.

For the barrier to the right of $t=t_0$ we can repeat the same procedure. Suppose we have a feasible solution with $q_{i,t}>0$ and $p_{i,s}>0$ for some $i$ and $t_0<t<s$. We take some $0<\eps<\min\{q_{i,t}, \shalf p_{i,s}\}$ and again transfer $\eps$ of the mass that currently leaves $(i,s)$ onto $(i,t)$. To do this we use the same $\tilde{p}$ as above, and define $\hat{p},\hat{q}$ to be the values after the transfer by
\begin{align*}
&\hat{p}_{j,r}=p_{j,r}, \quad &&\hat{q}_{j,r}=q_{j,r}, \quad &&\forall j,r<t \\
&\hat{p}_{i,t}=p_{i,t}+\tilde{p}_{i,s}=p_{i,t}+\eps, \quad &&\hat{q}_{i,t}=q_{i,t}-\eps \\
&\hat{p}_{j,t}=p_{j,t}, \quad &&\hat{q}_{j,t}=q_{j,t}, \quad &&\forall j\neq i \\
&\hat{p}_{j,r}=p_{j,r}+\tilde{p}_{j,r+s-t}, \quad &&\hat{q}_{j,r}=q_{j,r}+\tilde{q}_{j,r+s-t}, \quad &&\forall j,t<r<s \\
&\hat{p}_{i,s}=p_{i,s}+\tilde{p}_{i,2s-t}-\eps, \quad &&\hat{q}_{i,s}=q_{i,s}+\tilde{q}_{i,2s-t}+\eps, \\
&\hat{p}_{j,s}=p_{j,s}+\tilde{p}_{j,2s-t}, \quad &&\hat{q}_{j,s}=q_{j,s}+\tilde{q}_{j,2s-t}, \quad &&\forall j\neq i \\
&\hat{p}_{j,r}=p_{j,r}+\tilde{p}_{j,r+s-t}-\tilde{p}_{j,r}, \quad &&\hat{q}_{j,r}=q_{j,r}+\tilde{q}_{j,r+s-t}-\tilde{q}_{j,r}, \quad &&\forall j,r>t.
\end{align*}
 Again, \thref{rostfeasible} and \thref{rostoptimal} show that these are $\mathcal{P}^N$-feasible and that if we have such points $(i,t),(i,s)$ then we can improve our payoff by moving mass. Therefore, if $q_{i,t}>0$ for some $(i,t)$ where $t_0<t$, then $p_{i,s}=0$ $\forall t<s$.
\end{proof}

\begin{lemma} \thlabel{rostfeasible}
The $\bar{p},\hat{p}$ defined above are $\mathcal{P}^N$-feasible.
\end{lemma}
\begin{proof}

First note that
\begin{align*}
\tilde{p}_{j,r+1}+\tilde{q}_{j,r+1}&=(p_{j,r+1}+q_{j,r+1})\frac{\tilde{p}_{j+1,r}+\tilde{p}_{j-1,r}}{p_{j+1,r}+p_{j-1,r}}\\
&=\half (p_{j+1,r}+p_{j-1,r}) \frac{\tilde{p}_{j+1,r}+\tilde{p}_{j-1,r}}{p_{j+1,r}+p_{j-1,r}} \\
&= \half (\tilde{p}_{j+1,r}+\tilde{p}_{j-1,r}),
\end{align*}
and so the $\eps$ of mass evolves as expected. Also, since $\tilde{p}_{j,s} \geq 0$ $\forall j$, by induction we have that $\tilde{p}_{j,r},\tilde{q}_{j,r} \geq 0$ $\forall(j,r)$. Then, from the above, 
\begin{equation*}
\tilde{p}_{j,r+1} \leq \half (\tilde{p}_{j+1,r}+\tilde{p}_{j-1,r}), 
\end{equation*}
and we also know that $\tilde{p}_{i,s}=\eps<p_{i,s}$ and $\tilde{p}_{j,s}=0\leq p_{j,s}$ otherwise. Then, by induction,
\begin{equation*}
0\leq \tilde{p}_{j,r} \leq p_{j,r} \quad \forall (j,r),
\end{equation*}
and so
\begin{equation*}
0\leq \tilde{q}_{j,r} \leq q_{j,r} \frac{p_{j+1,r}+p_{j-1,r}}{p_{j+1,r}+p_{j-1,r}}= q_{j,r} \quad \forall (j,r).
\end{equation*}
It is then clear immediately that $\bar{p}_{j,r},\bar{q}_{j,r}\geq 0$ $\forall(j,r)\neq(i,s)$. Note that $\tilde{q}_{i,s}\leq \frac{\eps}{2}$ and therefore $\eps<\shalf q_{i,s}<\frac{2}{3}q_{i,s}\implies \tilde{q}_{i,s}+\eps\leq\frac{3}{2}\eps<q_{i,s}$, so we also have that $\bar{q}_{i,s}\geq 0$.

The new system also embeds the same amount of mass at every level, for example for $j\neq i$, we can check that our potential condition will not change:
\begin{align*}
\sum_{r} \bar{p}_{j,r}&= \sum_{r} p_{j,r} + \medspace \sum_{\mathclap{r>t}} \tilde{p}_{j,r-(t-s)} - \sum_{r>s} \tilde{p}_{j,r}\\
&= \sum_{r} p_{j,r} + \sum_{r>s} \tilde{p}_{j,r} - \sum_{r>s} \tilde{p}_{j,r}\\
&= \sum_{r} p_{j,r}.
\end{align*}
The $\hat{p}$ are similar.
\end{proof}

\begin{lemma} \thlabel{rostoptimal}
The new primal solution reduces our objective function:
\begin{align*}
\text{if } q_{i,t}>0 \text{ and }p_{i,s}>0 \text{ for some }i\text{ and }s<t<t_0 \text{ then }\sum f(r) \bar{q}_{j,r} \geq \sum f(r) q_{j,r}, \\
\text{if } q_{i,t}>0 \text{ and }p_{i,s}>0 \text{ for some }i\text{ and }t_0<t<s \text{ then }\sum f(r) \hat{q}_{j,r} \geq \sum f(r) q_{j,r}.
\end{align*}
\end{lemma}
\begin{proof}
Consider first the case of the inverse barrier. We have $\{r>s\}=\{r+t-s<t_0\}\cup\{r<t_0<r+t-s\}\cup\{t_0\leq r\}$ and in the first region,
\begin{equation*}
\bar{f}^N(r+t-s)-\bar{f}^N(t)>\bar{f}^N(r)-\bar{f}^N(s) \quad \text{for } r+t-s<t_0,
\end{equation*}
by the convexity of $\bar{f}^N$ on $[0,t_0]$. For $r<t_0<r+t-s$,
\begin{align*}
\bar{f}^N(r+t-s)-\bar{f}^N(t)&>\bar{f}^N(t_0)-\bar{f}^N(t)  &&\text{($t_0<r+t-s$, $\bar{f}^N$ strictly increasing on $[t_0,\infty)$)} \\
&>\bar{f}^N(t_0)-\bar{f}^N(s+t_0-r)  &&\text{($s+t_0-r<t$, $\bar{f}^N$ strictly decreasing on $[0,t_0]$)} \\
&>\bar{f}^N(r)-\bar{f}^N(s)  &&\text{($\bar{f}^N$ convex on $[0,t_0]$)}.
\end{align*}
Finally, for $t_0\leq r$,
\begin{align*}
\bar{f}^N(r+t-s)-\bar{f}^N(t)&>\bar{f}^N(r)-\bar{f}^N(t) \quad &&\text{($\bar{f}^N$ increasing on $[t_0,\infty)$)} \\
&>\bar{f}^N(r)-\bar{f}^N(s) \quad &&\text{($\bar{f}^N$ decreasing on $[0,t_0]$)}.
\end{align*}
Then by the definition of $\bar{q}$ in each of our time regions, we have
\begin{align*}
\sum_{j,r} \bar{f}(t) \bar{q}_{j,r} &= \sum_{j,r} \bar{f}(r) q_{j,r} +\eps (\bar{f}(s)-\bar{f}(t)) -\sum_{\mathclap{r>s,j}} \bar{f}(r) \tilde{q}_{j,r}  + \sum_{\mathclap{r>t,j}} \bar{f}(r)\tilde{q}_{j,r-(t-s)} \\
&=\sum_{j,r} \bar{f}(r) q_{j,r} +\eps (\bar{f}(s)-\bar{f}(t)) -\sum_{\mathclap{r>s,j}} \bar{f}(r) \tilde{q}_{j,r}  + \sum_{\mathclap{r>s,j}} \bar{f}(r+t-s)\tilde{q}_{j,r} \\
&=\sum_{j,r} \bar{f}(r) q_{j,r} +\eps (\bar{f}(s)-\bar{f}(t)) +\sum_{\mathclap{r>s,j}} (\bar{f}(r+t-s)-\bar{f}(r)) \tilde{q}_{j,r} \\
&= \sum_{j,r} \bar{f}(r) q_{j,r} +\sum_{\mathclap{r>s,j}} (\bar{f}(s)-\bar{f}(t))\tilde{q}_{j,r} +\sum_{\mathclap{r>s,j}} (\bar{f}(r+t-s)-\bar{f}(r)) \tilde{q}_{j,r}\\
&= \sum_{j,r} \bar{f}(r) q_{j,r} +\sum_{\mathclap{r>s,j}} (\bar{f}(r+t-s)-\bar{f}(t))\tilde{q}_{j,r} -\sum_{\mathclap{r>s,j}} (\bar{f}(r)-\bar{f}(s)) \tilde{q}_{j,r}\\
&\geq \sum_{j,r} \bar{f}(r) q_{j,r}.
\end{align*}

Briefly, here is how we get the above. For the first line we have just written out $\bar{p},\bar{q}$ in full. We know that we almost surely embed all of the released mass in finite time and so $\sum_{j,r} \tilde{q}_{j,r} =\eps$, and this gives us the fourth equality above. Finally, as we noted above, $\bar{f}^N(r+t-s)-\bar{f}^N(t)>\bar{f}^N(r)-\bar{f}^N(s)$ for all $r>s$.

We can also show that this argument doesn't change if we add a linear function of time to our payoff, as mentioned in \thref{Rostembed}. Considering now $\bar{f}(t)+C\frac{t}{N}$ we have, in the notation above, $C\sum_{r>s,j} \frac{r}{N}\tilde{q}_{j,r}-C\sum_{r>t,j} \frac{r}{N}\tilde{q}_{j,r-(t-s)}=C\eps\frac{t-s}{N}$, so,
\begin{align*}
\sum_{j,r} (\bar{f}(r)+C\frac{r}{N}) \bar{q}_{j,r} &= \sum_{j,r} (\bar{f}(r)+C\frac{r}{N}) q_{j,r} +\eps (\bar{f}(s)-\bar{f}(t)) +C\eps\frac{s-t}{N}-\sum_{\mathclap{r>s,j}} (\bar{f}(r)+C\frac{r}{N}) \tilde{q}_{j,r}  \\
& \qquad \qquad  \qquad \qquad \qquad \qquad \qquad \qquad \qquad \qquad+ \sum_{\mathclap{r>t,j}} (\bar{f}(r)+C\frac{r}{N}) \tilde{q}_{j,r-(t-s)} \\
&=\sum_{j,r} (\bar{f}(r)+C\frac{r}{N}) q_{j,r} +\eps (\bar{f}(s)-\bar{f}(t)) -\sum_{\mathclap{r>s,j}} \bar{f}(r) \tilde{q}_{j,r}  + \sum_{\mathclap{r>t,j}} \bar{f}(r) \tilde{q}_{j,r-(t-s)} \\
&\geq \sum_{j,r} (\bar{f}(r)+C\frac{r}{N}) q_{j,r}.
\end{align*}

For the case of the right-hand barrier,
\begin{align*}
\sum_{j,r} \bar{f}(t) \hat{q}_{j,r} &= \sum_{j,r} \bar{f}(r) q_{j,r} +\eps (\bar{f}(s)-\bar{f}(t)) -\sum_{\mathclap{r>s,j}} \bar{f}(r) \tilde{q}_{j,r}  + \sum_{\mathclap{r>t,j}} \bar{f}(r)\tilde{q}_{j,r-(t-s)} \\
&=\sum_{j,r} \bar{f}(r) q_{j,r} +\eps (\bar{f}(s)-\bar{f}(t)) -\sum_{\mathclap{r>s,j}} \bar{f}(r) \tilde{q}_{j,r}  + \sum_{\mathclap{r>s,j}} \bar{f}(r+t-s)\tilde{q}_{j,r} \\
&=\sum_{j,r} \bar{f}(r) q_{j,r} +\eps (\bar{f}(s)-\bar{f}(t)) +\sum_{\mathclap{r>s,j}} (\bar{f}(r+t-s)-\bar{f}(r)) \tilde{q}_{j,r} \\
&\geq \sum_{j,r} \bar{f}(r) q_{j,r},
\end{align*}
since $\bar{f}(r)$ is increasing for $r>t_0$, and similarly when we add a linear term.
\end{proof}

\begin{remark}
For the corresponding results in the cases where $\bar{F}^N_{j,t}=f(\frac{t}{N})$ for $f$ convex or concave, just consider the Rost or Root parts of the above respectively and the arguments hold exactly.
\end{remark}

\begin{remark}
Similarly to \thref{Rostembed}, if $F(x,t)=f(t)$ and $f'(t)$ is bounded, then we can consider $f(t)-Ct$ for some large constant $C$ and show that actually $p_{j,t}q_{j,t}=0$ for all $(j,t)$. This is not immediately obvious from \thref{caveshape} since we could have finitely many points at which $p_{j,t}>0,q_{j,t}>0$. Since these points will disappear when we take limits, this isn't an important detail and we omit the proof. 
\end{remark}

\section{Convergence of the Discrete Problem}
\subsection{Recovering the Continuous Optimiser}
Now we have a full picture for the discrete problem, we can show that the sequence of linear programming problems indeed converges to our continuous problem. We first show that if we discretise the optimal continuous time solution, then we recover it in the limit of the discrete problems produced. 

Let $\tau$ be a solution of \eqref{OptSEP} and $\tilde{\tau}^N$ the corresponding stopping time of the random walk $Y^N_t$ as defined earlier to be the time $k$ such that $\tau_{k-1}^N<\tau\leq\tau^N_k$. 

\begin{theorem} \thlabel{convergence}
For a function $F(x,t)$ continuous in both variables, and a suitable discretisation $\bar{F}^N(j,t)$ chosen so that $\bar{F}^N(\lfloor \sqrt{N}x\rfloor,\lfloor Nt\rfloor)\rightarrow F(x,t)$, we have 
\begin{equation*}
\bar{F}^N\left(\sqrt{N}Y^N_{\ttau^N},\ttau^N\right)\xrightarrow{\P} F(W_{\tau}, \tau)\quad \text{ as }N\rightarrow\infty.
\end{equation*}
In particular, when $F$ is bounded,
\begin{equation*}
\Ep{\bar{F}^N\left(\sqrt{N}Y^N_{\ttau^N},\ttau^N\right)}\rightarrow \Ep{F(W_{\tau}, \tau)}.
\end{equation*}
\end{theorem}
To prove this we use the following lemma.

\begin{lemma} \thlabel{convlemma}
For the $\tilde{\tau}^N$ defined above, we have
\begin{alignat*}{3}
Y^N_{\ttau^N}&\rightarrow W_{\tau} \quad &&\text{almost surely, and}\\
\frac{\ttau^N}{N}&\rightarrow\tau \quad &&\text{in probability, as } N\rightarrow\infty.
\end{alignat*}
\end{lemma}

\begin{proof}
Note that for any $\omega$, $\left|W_{\tau}(\omega)-Y^N_{\ttau^N}(\omega)\right|=\left|W_{\tau}(\omega)-W_{\tau^N_{\ttau}}(\omega)\right|<\frac{2}{\sqrt{N}}$, and so $Y^N_{\ttau^N}\rightarrow W_{\tau}$ almost surely. 

For the stopping time convergence, let $M(t)=\sup\{k:\tau^N_k\leq T\}$ for any $t$. We claim that for any $T>0, \eps>0$,
\begin{equation*}
\Prob{\sup_{s\leq T}\left|\frac{M(s)}{N}-s\right|>\eps}\rightarrow 0, \text{ as } N\rightarrow\infty.
\end{equation*}
If this is true, then $\frac{M(\tau)}{N}\xrightarrow{\P}\tau$ as $N\rightarrow\infty$, but since $M(\tau)=\ttau-1$ we have our result. All that remains is to prove the claim. 

Fix $T,\eps>0$ and let $X_n=\frac{n}{N}-\tau^N_n$, so $(X_n)_n$ is a martingale. If $n_0$ is such that $M(T)\leq n_0$, then
\begin{align*}
\sup_{s\leq T}\left(\frac{M(s)}{N}-s\right)&\leq \sup_{n\leq n_0} \left\{ \tau^N_n-\frac{n-1}{N}\right\}\leq \sup_{n\leq n_0} \left\{\frac{1}{N} -X_n \right\}, \\
\sup_{s\leq T}\left(s-\frac{M(s)}{N}\right)&\leq \sup_{n\leq n_0} \left\{ \frac{n}{N}-\tau^N_n\right\}=\sup_{n\leq n_0} X_n.
\end{align*}
Therefore,
\begin{equation*}
\sup_{s\leq T}\left|\frac{M(s)}{N}-s\right|\leq \frac{1}{N} +\sup_{n\leq n_0} |X_n|.
\end{equation*}
Choose $n_0=2TN^{\frac{3}{2}}$, then $\Prob{M(T)>n_0}<\frac{\eps}{3}$ for all sufficiently large $N$. Also, $X_{n_0}=\sum_{k=1}^{n_0} X_k-X_{k-1}$, the sum of iid mean zero random variables with variance $\frac{2}{3N^2}$. Then, by Doob's martingale inequality,
\begin{equation*}
\Prob{\sup_{n\leq n_0} |X_n|\geq\eps}\leq \frac{4T}{3\eps^2\sqrt{N}}<\frac{\eps}{3}, \text{ for large }N.
\end{equation*}
\end{proof}

Now we know how $Y^N_{\ttau^N}$ and $\ttau^N$ converge, we can prove our theorem.
\begin{proof}[Proof of \thref{convergence}]

\quad Since $\frac{\ttau^N}{N}\xrightarrow{\P} \tau$, and $Y^N_{\ttau^N}\rightarrow W_{\tau}$ almost surely, and $F$ is continuous in both variables, $F\left(Y^N_{\ttau^N}, \frac{\ttau^N}{N}\right)\xrightarrow{\P} F(W_{\tau}, \tau)$ as $N\rightarrow\infty$. Now,
\begin{align*}
\left| \bar{F}^N\left(\sqrt{N}Y^N_{\ttau^N}, \ttau^N\right)- F(W_{\tau}, \tau)\right|&\leq \left|\bar{F}^N\left(\sqrt{N}Y^N_{\ttau^N}, \ttau^N\right)- F\left(Y^N_{\ttau^N}, \frac{\ttau^N}{N}\right)\right| \\
&\qquad\qquad\qquad\qquad\qquad\qquad\qquad+ \left|F\left(Y^N_{\ttau^N}, \frac{\ttau^N}{N}\right)- F(W_{\tau},\tau)\right|,
\end{align*}
but by the convergence of $\bar{F}^N$, for any $\omega$ and any $\eps>0 \medspace \exists N_{\eps,\omega}$ such that $\forall N\geq N_{\eps,\omega}$, 
\begin{equation*}
\left|\bar{F}^N\left(\sqrt{N}Y^N_{\ttau^N}, \ttau^N\right)(\omega)- F\left(Y^N_{\ttau^N}, \frac{\ttau^N}{N}\right)(\omega)\right|<\eps,
\end{equation*}
and therefore 
\begin{equation*}
\bar{F}^N\left(\sqrt{N}Y^N_{\ttau^N},\ttau^N\right)\xrightarrow{\P} F(W_{\tau}, \tau)\quad \text{ as }N\rightarrow\infty.
\end{equation*}
\end{proof}

\begin{remark}\thlabel{greaterthan}
This result tells us that if we discretise the optimal continuous time solution, to get some feasible $p^{\tau,N}_{j,t}$, and then take the limit, we recover our optimal value, so in particular,
\begin{equation*}
\mathrm{P}^N\geq \sum_{j,t} \bar{F}^N_{j,t} q^{\tau,N}_{j,t} \implies \lim_{N\rightarrow\infty} \mathrm{P}^N\geq \lim_{N\rightarrow\infty}\sum_{j,t} \bar{F}^N_{j,t} q^{\tau,N}_{j,t}= \Ep{F(W_{\tau}, \tau)}.
\end{equation*}
\end{remark}


\begin{remark}
Note that as a corollary of \thref{convlemma} we have that $\mu^N=\mathcal{L}\left(Y^N_{\ttau^N}\right)\rightarrow\mathcal{L}\left(W_{\tau}\right)=\mu$.
\end{remark}

\subsection{Convergence of Barriers}
We know that in the case of a cave, Root, or Rost payoff, for each $N$, $\mathrm{P}^N$ is attained by some $p^{*,N}$ which give us a barrier-type property for a random walk. Consider again the cave case, then we know that for each $j$ there is a largest time $\bar{l}^N_j<t_0$ such that $p_{j,t}=0$ $\forall t\leq\bar{l}^N_j$, and similarly a smallest time $\bar{r}^N_j>t_0$ such that $p_{j,t}=0$ $\forall t\geq \bar{r}^N_j$. For $i=0,L$ we take $\bar{l}^N_{j^N_i}=\bar{r}^N_{j^N_i}$ to include the upper and lower boundaries. Denote this stopping region by $\hat{\mathcal{B}}^N$. Note that for each $j$ we either have $q_{j,\bar{l}^N_j}>0$, or $q_{j,s}=0$ $\forall s<t_0$, and similarly for $\bar{r}^N_j$. To find the corresponding stopping region for the Brownian motion in continuous time, we shift from discrete to continuous time, so let
\begin{equation*}
\mathcal{B}^N=\left\{(x,t):(x,\lfloor Nt\rfloor)\in\hat{\mathcal{B}}^N\right\}=\left\{(x,t): \medspace t=0,\infty \text{ or } t\in[0,\frac{\bar{l}^N_j}{N}]\cup[\frac{\bar{r}^N_j}{N},\infty], x=x_j^N,\text{ for some }j\right\}.
\end{equation*}
Let $\bar{\tau}^{N,n}=\inf\{t\geq0:\medspace (Y^N_t,t)\in\hat{\mathcal{B}}^n\}$ and write $\bar{\tau}^N:=\bar{\tau}^{N,N}$, with Brownian equivalent $\tau^N=\inf\{t\geq0:\medspace (W_t,t)\in\mathcal{B}^N\}$.

We now show that the stopped random walks converge to a stopped Brownian motion with the correct distribution, and therefore that $\mathrm{P}^N$ converges to the continuous time optimal value.

\begin{lemma} \thlabel{barrierconv}
  The cave barriers $\mathcal{B}^N$ converge (possibly along a subsequence) to another cave barrier $\mathcal{B}^{\infty}$, and $(W_{\tau^N},\tau^N)\xrightarrow{\P}(W_{\tau^{\infty}},\tau^{\infty})$ as $N\rightarrow\infty$, where $\tau^{\infty}$ is the Brownian hitting time of $\mathcal{B}^{\infty}$.
\end{lemma}

\begin{proof}
  Define a sequence of measures $\rho_N$ on $[x_*,x^*]\times[0,\infty]$ by $\rho_N(\cdot)=\Prob{(W_{\tau^N},\tau^N)\in \cdot}$. Since $\tau^N\leq H_{x_*}\wedge H_{x^*}$ for all $N$, and $H_{x_*}\wedge H_{x^*}$ is an integrable stopping time, we know that $\forall \eps>0 \, \exists y_{\eps}$ such that $\Prob{\tau^N>y_{\eps}}<\eps \medspace \forall N$, and therefore there is a compact set $A\subseteq [x_*,x^*]\times\R_+$ such that $\rho_N(A)<\eps$ $\forall N$. In particular, the sequence $(\rho_N)$ is tight, and then by Prokhorov's theorem, there exists some $\rho_{\infty}$ such that $\rho_N\xrightarrow{w}\rho_{\infty}$ (perhaps after restricting to a suitable subsequence). What remains to show is that $\rho_{\infty}(\cdot)=\Prob{(W_{\tau^{\infty}},\tau^{\infty})\in\cdot}$, and we follow the ideas of \cite{Root:1969aa} and \cite{Cox:2015aa}.

In \cite{Root:1969aa}, Root maps the closed half plane onto a closed, bounded rectangle and defines a norm on the space of closed subsets of the half plane by $d(R,S)=\sup_{x\in R} \inf_{y\in S} r(x,y)$, where $r$ is the metric induced by taking the Euclidean metric on the rectangle. Under $d$, the space of closed subsets of our half plane is a separable, compact metric space and the space of all cave barriers is a closed subspace of the space, so is compact. We have a sequence of regions $\mathcal{B}^N$, and then by compactness they converge (possibly after taking a further subsequence) to some cave barrier $\mathcal{B}^{\infty}$ in this norm. Denote the hitting time of $\mathcal{B}^{\infty}$ by $\tau^{\infty}$.

Consider first our Root barrier, and let $\mathcal{B}^N$ now just denote the barrier part of the stopping region. By \cite{Root:1969aa} we then know that the hitting times of the $\mathcal{B}^N$ converge to the hitting time of $\mathcal{B}^{\infty}$ in probability as $N\rightarrow\infty$, and therefore $W_{\tau^N}\xrightarrow{\P} W_{\tau^{\infty}}$ also (for example by considering $\tau^N\wedge\tau^{\infty}$ and $\tau^N\vee \tau^{\infty}$).

Now we consider $\mathcal{B}^N$ to be just the Rost inverse-barrier section of our stopping region. In \cite{Cox:2015aa} the authors define the Rost inverse-barrier by curves $b:(0,\infty)\rightarrow \R\cup\{+\infty\}$ and $c:(0,\infty)\rightarrow\R\cup\{-\infty\}$ so that the Rost stopping time is $\tau_{b,c}=\inf\left\{ t>0: \medspace W_t\geq b(t)\text{ or }W_t\leq c(t) \right\}$. We can also define our atomic stopping region like this for each $N$, giving a sequence of curves $b^N$, and apply the same arguments. All we need to show is that $\forall \eps>0$, $\exists n$ such that $\forall N\geq n$, $b^{\eps}\geq b^N \geq b_{\eps}$, where $b^{\eps}(t)=b(t+\eps)+\eps$ and $b_{\eps}(t)=b(t-\eps)-\eps$ are defined in \cite{Cox:2015aa}. We know that each $\mathcal{B}^N$ has absorbing boundaries, and this must hold in the limit, i.e. $b^N(t)=b^{\infty}(t)=x^*$ $\forall t\geq K(x^*)$ and similarly for $c$, where $b^{\infty}$ is the curve associated to $\mathcal{B}^{\infty}$. This means we can just consider the part of the Rost inverse-barriers that are in the compact region to the left of the curve $K(x)$, but then in this region we have that the metric $r$ is bounded, and so for any $\delta>0$, we can find an $\eps>0$ such that $d(\mathcal{B}^N,\mathcal{B})<\delta \implies b^{\eps}\geq b^N \geq b_{\eps}$, with $d$ the above Root norm.

Then, combining the results of \cite{Root:1969aa} and \cite{Cox:2015aa} by considering the minimum of the two stopping times, we have that $(W_{\tau^N},\tau^N)\xrightarrow{\P}(W_{\tau^{\infty}},\tau^{\infty})$ as $N\rightarrow\infty$. 
\end{proof}

\begin{lemma} \thlabel{rwbarrierconv}
With the stopping times defined previously, 
\begin{equation*}
\left| \left( Y^N_{\bar{\tau}^{N}},\frac{\bar{\tau}^{N}}{N}\right) - \left( W_{\tau^N},\tau^N\right) \right| \xrightarrow{d} 0  \quad \text{as } N\rightarrow\infty.
\end{equation*}
\end{lemma}
\begin{proof}
We again consider the barrier and inverse-barrier hitting times separately and use the ideas of  \cite{Root:1969aa} and \cite{Cox:2015aa}.

For the Rost inverse-barrier alone we consider the Girsanov Theorem approach of \cite{Cox:2015aa}. Fix $\delta>0$ and some $T>0$ and let $\eps=\frac{\delta}{8x^*}$. Donsker's theorem tells us that for any fixed $T$ there is a Brownian motion $B$ such that $\P\left(\sup_{0\leq t\leq T} \left|B_t-Y^N_{\lfloor Nt \rfloor}\right| >\eps\right)\rightarrow0$ as $N\rightarrow\infty$. In particular, if $A^{N,\eps}=\{\sup_{0\leq t\leq T} \left|B_t-Y^N_{\lfloor Nt \rfloor}\right| \leq\eps\}$ then $\exists N_0$ such that $N\geq N_0 \implies \P\left(\left(A^{N,\eps}\right)^C\right)<\frac{\delta}{4}$. Let $B^{\pm\eps}_t=B_t\pm \eps t$ and denote the associated hitting times of $\mathcal{B}^N$ by $\tau^{N,\pm \eps}$. The Girsanov Theorem tells us that $B^{\eps}$ is a $\Q_{\eps}$-Brownian motion, where $\frac{\di \Q_{\eps}}{\di \P}\Big|_{\F_t}=\exp\left(\eps B_t-\half \eps^2 t\right)$, and similarly for $B^{-\eps}$. Also, on $A_{N,\eps}$, we have $\tau^{N,\eps}\leq \tau^N, \frac{\bar{\tau}^N}{N}\leq \tau^{N,\eps}$. For any $t\leq T$, $N\geq N_0$,
\begin{align*}
\left| \Prob{N^{-1}\bar{\tau}^N\geq t} -\Prob{\tau^N\geq t} \right|&\leq\left| \Probc{N^{-1}\bar{\tau}^N\geq t}{A^{N,\eps}} -\Probc{N^{-1}\tau^N\geq t}{A^{N,\eps}}\right|\Prob{A^{N,\eps}} \\
& \quad + \left| \Probc{N^{-1}\bar{\tau}^N\geq t}{\left(A^{N,\eps}\right)^C} -\Probc{\tau^N\geq t}{\left(A^{N,\eps}\right)^C}\right|\Prob{\left(A^{N,\eps}\right)^C} \\
&\leq \left| \Probc{N^{-1}\bar{\tau}^N\geq t}{A^{N,\eps}} -\Probc{\tau^N\geq t}{A^{N,\eps}}\right| +2 \Prob{\left(A^{N,\eps}\right)^C} \\
&\leq \frac{\delta}{2} +\left| \Prob{\tau^{N,\eps}\geq t} -\Prob{\tau^{N,-\eps}\geq t} \right| \\
&=\frac{\delta}{2} +\left| \Eps{\eps}{\iden\{\tau^N\geq t\}} - \Eps{-\eps}{\iden\{\tau^N\geq t\}} \right| \\
&=\frac{\delta}{2} +\left| \Ep{\left(\left.\frac{\di \Q_{\eps}}{\di \P}-\frac{\di \Q_{-\eps}}{\di \P}\right)\right|_{\F_{\tau^N}}\iden\{\tau^N\geq t\}} \right| \\
&=\frac{\delta}{2} +\left| \Ep{\e^{\eps B_{\tau^N} -\frac{1}{2}\eps^2 {\tau^N}}\left(1-\e^{-2\eps B_{\tau^N}}\right)\iden\{\tau^N\geq t\}} \right| \\
&\leq \frac{\delta}{2} + \e^{\eps x^*}\left(1-\e^{-2\eps x^*}\right) \\
&\leq \frac{\delta}{2} + 2\eps x^* \\
&< \delta.
\end{align*}
It follows that $\left|\frac{\bar{\tau}^N}{N}-\tau^N\right|\xrightarrow{d}0$ as $N\rightarrow\infty$, and therefore $\left|\frac{\bar{\tau}^N}{N}-\tau^N\right|\xrightarrow{\P}0$ since the limit is a constant. Then, $\left| Y^N_{\bar{\tau}^{N}}-W_{\tau^N}\right|\xrightarrow{\P}0$ also as $N\rightarrow\infty$, and therefore for the Rost part of the stopping time we have the required convergence in probability. To see that $\left| Y^N_{\bar{\tau}^{N}}-W_{\tau^N}\right|\xrightarrow{\P}0$, note that $\left| Y^N_{\bar{\tau}^{N}}-W_{\tau^N}\right|\leq \left| Y^N_{\bar{\tau}^{N}}-W_{\frac{\bar{\tau}^{N}}{N}}\right|+\left|W_{\frac{\bar{\tau}^{N}}{N}}-W_{\tau^N}\right|$ and use Donsker's theorem on the first term and the convergence of the stopping times in the second term.

Now consider just the Root barrier and fix $\eps>0$. We repeat the proof of \cite[Lemma~2.4]{Root:1969aa} to show that for any $\eps>0$, $\exists N_0,n_0$ such that $ N\geq N_0,n\geq n_0 \implies \Prob{\left|\bar{\tau}^{N,n}-\bar{\tau}^N\right|>\eps}<\eps$ and therefore $\left|\bar{\tau}^{N}-\bar{\tau}^{N,n}\right|\xrightarrow{\P}0$ as $N\rightarrow\infty$. Fix $\eps>0$. Again by Donsker's Theorem, $\exists N_0$ such that $N\geq N_0 \implies \P\left(A^{N,\eps}\right)>1-\frac{\eps}{6}$. We can choose $\eta>0$ such that $\P\left(\sup_{\eta<t<\eps} B_t-\eps>\eta \text{ and } \inf_{\eta<t<\eps} B_t+\eps>-\eta\right)>1-\frac{\eps}{6}$, and therefore for $N\geq N_0$,
\begin{align*}
\P\left(\sup_{\eta<t<\eps} Y^N_{Nt}>\eta\sqrt{N} \text{ and } \inf_{\eta<t<\eps} Y^N_{Nt}>-\eta\sqrt{N}\right)&\geq \P\left(A^{N,\eps}\right)\P\left(\sup_{\eta<t<\eps} B_t-\eps>\eta \text{ and } \inf_{\eta<t<\eps} B_t+\eps>-\eta\right) \\
&>1-\frac{\eps}{3}.
\end{align*}
Let $\tau^{\pi}=H_{x_{j^N_0}}\wedge H_{x_{J^N_L}}$ and $\bar{\tau}^{N,\pi}=H^N_0\wedge H^N_L$. Since the hitting region with atoms only at $j=j^N_0,j^N_L$ is an example of a Rost barrier, we know from the above that $\left(Y^N_{\bar{\tau}^{N,\pi}},\frac{\bar{\tau}^{N,\pi}}{N}\right)\xrightarrow{d}\left(W_{\tau^{\pi}},\tau^{\pi}\right)$ as $N\rightarrow\infty$, and so since we are working on a bounded domain, $\Ep{(Y^N_{\bar{\tau}^{N,\pi}})^2}\rightarrow\Ep{W_{\tau^{\pi}}^2}$. Then,
\begin{equation*}
\Ep{\bar{\tau}^N}\leq \Ep{\bar{\tau}^{N,\pi}}=\Ep{(Y^N_{\bar{\tau}^{N,\pi}})^2}\rightarrow\Ep{W_{\tau^{\pi}}^2},
\end{equation*}
as $N\rightarrow\infty$, and also each $\Ep{\bar{\tau}^{N,\pi}}$ is bounded. We can therefore find a uniform bound on the $\Ep{\bar{\tau}^N}$, and in particular $\exists T$ such that, by the Markov Inequality,
\begin{equation*}
\Prob{\bar{\tau}^N\geq T}\leq \frac{\Ep{\bar{\tau}^N}}{T}<1-\frac{\eps}{3} \quad \forall N.
\end{equation*}
We can also find $N\geq N_0, n\geq n_0$ such that $d\left(\hat{\mathcal{B}}^N,\hat{\mathcal{B}}^{n}\right)<\eta$, and then can follow exactly the argument of \cite[Lemma~2.4]{Root:1969aa} to get that $\left|\bar{\tau}^{N}-\bar{\tau}^{N,n}\right|\xrightarrow{\P}0$ as $N,n\rightarrow\infty$. We can use a similar argument to the above to then show that $\left|Y^N_{\bar{\tau}^{N}}-Y^N_{\bar{\tau}^{N,n}}\right|\xrightarrow{\P}0$, and so $\left|\left(Y^N_{\bar{\tau}^{N}},\bar{\tau}^{N}\right)-\left(Y^N_{\bar{\tau}^{N,n}},\bar{\tau}^{N,n}\right)\right|\xrightarrow{\P}0$ as $N,n\rightarrow\infty$. Donsker's theorem also shows that for any $n$, $\left(\frac{\bar{\tau}^{N,n}}{N}, Y^N_{{\bar{\tau}}^{N,n}}\right)\xrightarrow{d}\left(\tau^n, W_{\tau^n}\right)$ as $N\rightarrow\infty$, and we prove this in \thref{donskconv}. Combining these results and \thref{barrierconv} gives the necessary convergence.
\end{proof}

\begin{lemma} \thlabel{donskconv}
If $\tau^n$ and $\bar{\tau}^{N,n}$ are as defined above, then 
\begin{equation*}
\left(\frac{\bar{\tau}^{N,n}}{N}, Y^N_{{\bar{\tau}}^{N,n}}\right)\xrightarrow{d}\left(\tau^n, W_{\tau^n}\right) \quad \text{as } N\rightarrow\infty.
\end{equation*}
\end{lemma}

\begin{proof}
By the choice of our discretisation, we know by Donsker's Theorem that $(Y^N_{\lfloor Nt \rfloor}; t\leq T)\xrightarrow{d} (W_t;t\leq T)$ as $N\rightarrow \infty$ for any $T>0$, and then by the Portmanteau Theorem,
\begin{equation*}
\lim_{N\rightarrow\infty} \Prob{(Y^N_{Nt})_{t\leq T}\in \mathcal{K}}=\Prob{(W_t)_{t\leq T}\in \mathcal{K}} \quad \text{for all Borel } \mathcal{K} \text{ with } \Prob{X\in \partial\mathcal{K}}=0,
\end{equation*}
where we consider $\mathcal{K}$ to be a subset of the set of continuous paths $f\in \mathcal{C}[0,T]$. 

Fix $n$ and consider $Y^N_{\bar{\tau}^{N,n}}$ and $W_{\tau^n}$ for $N\geq n$. Take some point $(t, x_i^n)\in\mathcal{B}^n$ and fix a closed interval, $\bar{B}(\eps)=\{x_i^n\}\times[t-\gamma,t+\gamma]$, of width $\gamma<n^{-1}$ around this point. We show that the set of continuous paths which hit $\mathcal{B}^n$ for the first time in $\bar{B}(\eps)$ is a Borel set. Note that since we are working with discrete barriers, there is a smallest $y>x_i^n$ such that $\left(\{y\}\times[t-\gamma,t+\gamma]\right)\cap \mathcal{B}^n\neq \emptyset$, and also a largest $z<x_i^n$ satisfying the same property. Now consider the sets
\begin{alignat*}{3}
&\mathcal{K}^{\eps}_q&&:=\left\{ f\in \mathcal{C}[0,T]:\medspace f(s)<y-\eps \medspace \forall s\in[t-\gamma,q]\cap\Q\right\}, &\\
&\mathcal{K}^{\eps,\delta}_q&&:=\left\{ f\in \mathcal{C}[0,T]:\medspace f(q)<x_i^n+\delta\right\}\cap\mathcal{K}^{\eps}_q, \\
&\mathcal{K}^{\eps,\delta}&&:=\bigcup_{\substack{q\in[t-\gamma,t+\gamma] \\ q\in\Q}} \mathcal{K}^{\eps,\delta}_q, \\
&\mathcal{K}^y&&:=\bigcap_{\substack{\delta>0 \\ \delta\in\Q}} \bigcup_{\substack{\eps>0 \\ \eps\in\Q}} \mathcal{K}^{\eps,\delta}.
\end{alignat*}
Then, since $\mathcal{K}^{\eps}_q,\mathcal{K}^{\eps,\delta}_q$ are Borel, $\mathcal{K}^y$ is also. Similarly we can define the above when considering $z$ instead of $y$, with the opposite inequalities, and we would find that $\mathcal{K}^z$ is Borel. Since our barrier is a closed region, $\mathcal{K}^1:=\left\{ f\in \mathcal{C}[0,T]:\medspace f \text{ doesn't hit $\mathcal{B}^n$ before time } t-\gamma \right\}$ is open in $\mathcal{C}[0,T]$, and therefore $\mathcal{K}:=\mathcal{K}^1\cap\left(\mathcal{K}^y \cup \mathcal{K}^z\right)$ is a Borel set. But $\mathcal{K}$ is exactly the set of paths which hit $\mathcal{B}^n$ for the first time in $\bar{B}(\eps)$.

Now, $\partial\mathcal{K}$ is the set of paths which start at $0$ and either hit $\bar{B}(\eps)$ at times $t\pm \gamma$, or hit $\bar{B}(\eps)$ anywhere but also hit $\mathcal{B}^n$ elsewhere first without passing through any atoms of $\mathcal{B}^n$. The probability a Brownian path hits at $t\pm \gamma$ or one of the finite number of end points of the atoms is 0, but if it touches an atom of the boundary elsewhere then it will almost surely pass through the atom. Therefore, $\Prob{W\in\partial\mathcal{K}}=0$, so the Portmanteau Theorem applies. By the definition of $\mathcal{K}$, and since $\gamma,i$ were arbitrary, 
\begin{equation*}
\left(\frac{\bar{\tau}^{N,n}}{N}, Y^N_{\bar{\tau}^{N,n}}\right)\xrightarrow{d}\left(\tau^n, W_{\tau^n}\right) \quad \text{as } N\rightarrow\infty, \text{ for any }n.
\end{equation*}
\end{proof}

\begin{theorem} \thlabel{primalconv}
If $\mathrm{P}^N$ is our discrete optimal value, and $\tau$ is the optimal continuous time stopping time, then
\begin{equation*}
\mathrm{P}^N\rightarrow \Ep{F\left(W_{\tau},\tau\right)}, \quad \text{as }N\rightarrow\infty.
\end{equation*}
\end{theorem}

\begin{proof}
We know from our choice of $\bar{F}$ that $\| \bar{F}^N(x,Nt)-F(x,t) \|_{\infty}\rightarrow 0$ as $N\rightarrow\infty$. Then by \thref{rwbarrierconv} and the boundedness of $F$, we have
\begin{align*}
\Ep{\left|\bar{F}^N\left(\sqrt{N}Y^N_{\bar{\tau}^N},\bar{\tau}^N\right)-F\left(W_{\tau^N},\tau^N\right)\right|}&\leq\Ep{\| \bar{F}^N(x,Nt)-F(x,t) \|_{\infty}}\\
&\qquad\qquad\qquad\qquad + \Ep{\left|F\left(Y^N_{\bar{\tau}},\frac{\bar{\tau}^N}{N}\right)-F\left(W_{\tau^N},\tau^N\right)\right|} \\
&\rightarrow 0, \quad \text{as } N\rightarrow\infty.
\end{align*}
Also, \thref{barrierconv} shows that $(W_{\tau^N},\tau^N)\xrightarrow{\P}(W_{\tau^{\infty}},\tau^{\infty})$ as $N\rightarrow\infty$ and so, since $F$ is bounded,
\begin{equation*}
\Ep{F\left(W_{\tau^N},\tau^N\right)}\rightarrow\Ep{F\left(W_{\tau^{\infty}},\tau^{\infty}\right)} \quad \text{as } N\rightarrow\infty,
\end{equation*}
where $\tau^{\infty}$ is the hitting time of a cave barrier such that $W_{\tau^{\infty}}\sim\mu$, since $\mathcal{L}\left(W_{\tau^{\infty}}\right)=\lim_{N\rightarrow\infty}\mathcal{L}\left(Y^N_{\bar{\tau}^N}\right)\\=\lim_{N\rightarrow\infty}\mu^N=\mu$. Then, by the optimality of $\tau$, combining these results gives
\begin{equation*}
\mathrm{P}^N=\Ep{\bar{F}^N\left(\sqrt{N}Y^N_{\bar{\tau}^N},\bar{\tau}^N\right)}\rightarrow \Ep{F\left(W_{\tau^{\infty}},\tau^{\infty}\right)}\leq \Ep{F\left(W_{\tau},\tau\right)}.
\end{equation*}
\thref{convergence} gives the other inequality, as mentioned in \thref{greaterthan}.
\end{proof}

\section{Conclusions}
\thref{barrierconv} and \thref{primalconv} recover the cave embedding result proved using the monotonicity principle in \cite{Beiglboeck:2013aa}, that is, there exists a cave barrier such that the hitting time of that barrier minimises $\Ep{\varphi(\sigma)}$ over stopping times $\sigma$ such that $W_{\sigma}\sim\mu$. In addition, these results characterise these boundaries/stopping times as the limits of solutions to a discrete problem. The equivalent results hold in the cases of the Root and Rost embeddings, and also in the case of $K$-cave embeddings from \cite{Cox:2016aa}.

As well as this approach being a novel way of reproving the existence of these embeddings, it can also be used to derive properties of the continuous time problem which are not easily deduced otherwise. For example, our principal motivation for this work was to establish the form of the optimal superhedging portfolio of a European call option on a leveraged exchange traded fund, and without the work here it is not clear that such an optimal portfolio exists. In Section \ref{dualsection} we give an indication of how our dual optimisers $(\eta^{*,N},\nu^{*,N})$ converge to functions with which we can superhedge our payoff $F$, and this is formalised in \cite[Section~5.1]{Cox:2016aa}. Optimal superhedging portfolios are given for Root and Rost-type payoffs in \cite{Cox:2013ab} and \cite{Cox:2013ac} respectively, and we can also use this approach to recover those functions.

The discrete setup of this problem is robust in the sense that we can change the problem somewhat and hope to still prove strong duality and derive properties of the associated continuous time problem. By changing the conditions at $t=1$ in $\mathcal{P}^N$ we can consider the problem where our random walk starts according to some more general initial distribution. The strong duality and convergence results above will still hold provided we choose a starting measure (for the Brownian motion) $\mu_0$ such that $\mu_0\leq\mu$ in convex order, and choose the discretisations $\mu_0^N$ carefully.

Here we consider the case where the full distribution of the process at the terminal time is known, but this approach could give more insight into the problem where the potential of the terminal distribution is only given at finitely many points.

In the financial problem of finding robust bounds on the price of some exotic option, the Skorokhod embedding problem arises as a consequence of the Breeden-Litzenberger formula of \cite{Breeden:1978aa}. This result allows us to calculate the terminal distribution of a price process if we can observe the prices of European call options on this process at some fixed terminal time for all possible strikes. A more realistic assumption is that we can observe the call prices at finitely many strikes, and then we can only calculate the potential of the terminal distribution at these points.

\nocite{Morters:2010aa}
\nocite{Noschese:2013aa}
\nocite{Hobson:1998aa}
\bibliography{LETFnew.bib}
\bibliographystyle{abbrv}
\end{document}